\documentclass[11pt]{article}

\usepackage{amsmath,amsthm,amsfonts,amssymb,bm,bbm,enumerate, graphicx, mathtools,mathrsfs,color}
\usepackage{tcolorbox}
\usepackage[english]{babel}
\usepackage[utf8]{inputenc}
\usepackage{fancyhdr,longtable}
\usepackage{caption}
\usepackage{subcaption}
\usepackage[toc,page]{appendix}
\usepackage{csquotes,multirow}

\definecolor{NBrown}{HTML}{66220C}
\definecolor{NAqua}{HTML}{00698C}
\definecolor{ForestGreen}{HTML}{228b22}

\usepackage[backend=bibtex,style=ext-alphabetic,giveninits=true,maxnames=99,maxalphanames=99,autopunct=false]{biblatex}
\usepackage{hyperref}
\hypersetup{colorlinks=true, citecolor=ForestGreen, breaklinks, urlcolor=black, linkcolor=black}
\usepackage[capitalise]{cleveref}
\AtEveryCite{\color{ForestGreen}}

\parskip \medskipamount

\usepackage{tikz}

\usepackage{graphicx}
\usepackage{caption,subcaption}
\newtheorem{theorem}{Theorem}[section]
\newtheorem{lemma}[theorem]{Lemma}
\newtheorem{proposition}[theorem]{Proposition}
\newtheorem{corollary}[theorem]{Corollary}

\makeatletter
\renewcommand\@dotsep{10000}
\makeatother
\usepackage{mleftright}
\mleftright

\renewcommand{\epsilon}{\varepsilon}

\newcommand{\E}[1]{\mathbb{E}\!\left[#1\right]}

\newcommand{\estart}[2]{\mathbb{E}_{#2}\!\left[#1\right]}

\newcommand{\Pb}{\mathbb{P}}
\newcommand{\bPb}{\mathbf{P}}

\newcommand{\Pcal}{\mathcal{P}}
\newcommand{\indic}[1]{\mathbbm{1}_{\left\{#1\right\}}}

\newcommand{\R}{\mathbb{R}}

\newcommand{\N}{\mathbb{N}}

\newcommand{\I}{\mathcal{I}}
\newcommand{\T}{\mathcal{T}}

\newcommand{\F}{\mathcal{F}}
\newcommand{\W}{\mathcal{W}}
\newcommand{\Z}{\mathcal{Z}}
\newcommand{\J}{\hat{J}}

\newcommand{\ZZZ}{\overline{Z}}

\newcommand{\zetas}{\zeta^{[s]}}
\newcommand{\zetass}{\zeta^{[s^*]}}
\newcommand{\Wts}{\widehat{W}^{[s]}}
\newcommand{\Wtss}{\widehat{W}^{[s^*]}}
\newcommand{\Ws}{{W}^{[s]}}
\newcommand{\Wt}{\widehat{W}}
\newcommand{\Noo}{\N^{(1)}}
\newcommand{\No}{N^{(1)}}
\newcommand{\Ns}{N^{(s)}}
\def\X{X^{\text{exc}}}

\newcommand{\HH}{\overleftarrow{H}}

\def\Ta{\T_{\alpha}}

\newcommand{\dt}{d_{\alpha}}

\newcommand{\Ito}{It\^o }

\newcommand{\supp}{\text{supp }}
\def\Levy{L\'evy }

\def\cadlag{c\`{a}dl\`{a}g }

\allowdisplaybreaks

\begin{document}

\title{Some properties of stable snakes}
\author{{Eleanor Archer}\footnote{\href{mailto:earcher@parisnanterre.fr}{earcher@parisnanterre.fr}}, {Ariane Carrance}\footnote{\href{mailto:ariane.carrance@math.cnrs.fr}{ariane.carrance@math.cnrs.fr}},\, and {Laurent Ménard}\footnote{\href{mailto:laurent.menard@normalesup.org}{laurent.menard@normalesup.org}}}

\date{\today}

\maketitle


\begin{abstract}
We prove some technical results relating to the Brownian snake on a stable \Levy tree. This includes some estimates on the range of the snake, estimates on its occupation measure around its minimum and also a proof of the fact that the snake and the height function of the associated tree have no common increase points.


\end{abstract}



\section{Introduction}

The purpose of this work is to study the behaviour of the Brownian snake on a stable \Levy tree. Stable \Levy trees arise as scaling limits of critical Bienaymé--Galton--Watson trees with infinite variance offspring distributions, generalising the appearance of the Brownian CRT as the scaling limit of finite variance critical Bienaymé--Galton--Watson trees and were first studied by Le Gall and Le Jan in \cite{le1998branching}. The associated Brownian snake can be thought of as Brownian motion indexed by the \Levy tree and was studied by Duquesne and Le Gall in \cite{LeGDuqMono}. 

This work follows on naturally from the works of \cite{marzouk-brownianstable}, in which Marzouk shows that such snakes arise as universal scaling limits of discrete snakes on Bienaymé--Galton--Watson trees, and \cite{riera2022structure}, in which Riera and Rosales-Ortiz develop an excursion theory for such limiting snakes. One motivation for studying this model comes from close connections to superprocesses \cite{LeGDuqMono, LeGallBook}. Our main motivation stems from applications to the construction of scaling limits of random planar maps. Indeed, the Brownian snake (Brownian motion indexed by the Brownian CRT) is a crucial ingredient in the celebrated construction of the Brownian sphere and its identification as a universal scaling limit of a very large class of random planar map models \cite{marckert2006limit,le2013uniqueness, miermont2013brownian}. The work \cite{marzouk-brownianstable} suggests that a parallel construction of a \textit{stable sphere} could be made using the Brownian snake indexed by a stable \Levy tree, and, thanks to bijections between labelled trees and random planar maps, that this sphere should arise as the scaling limit of a class of stable-type quadrangulations (as well as other similar models). The purpose of this paper is to lay the groundwork for a programme to prove such a result.

The stable \Levy tree can be coded by a stable \Levy excursion in a canonical way (we will explain this precisely in Section \ref{section:stable-trees}). For this reason it is natural to consider the law of the associated snake under the \Ito excursion measure, which is the canonical law on such excursions, either in its unnormalised ($\N$) or normalised ($\Noo$) form.

Under $\N$ or $\Noo$, the Brownian snake on the stable \Levy tree can be encoded by a pair $(H_t, Z_t)_{t \in [0,\sigma]}$ where $\sigma$ denotes the length of the associated excursion (and is also known as the lifetime; it also corresponds to the total volume of the associated \Levy tree), $H$ denotes the height function of the associated tree and $Z_t$ denotes the value of the snake at the point in the tree coded by time $t$. Moreover, the pieces of notation $\N_x$ or $\Noo_x$ denote the fact that $Z_0=x$ (these notions will all also be defined precisely in Section \ref{section:stable-trees}).

In what follows, we let $\mathcal{R} = \{Z_t: t \in [0, \sigma]\}$ denote the range of the snake, and let $R = \sup \{x \in \R: x \in \mathcal{R}\}$ denote its maximal value. Our first result is a characterisation of the joint distribution of the lifetime and the range of the stable snake under \Ito excursion measure, much in the spirit of the work of Delmas~\cite{delmas2003computation} in the Brownian case:

\begin{theorem} \label{theorem:main}
Fix $(a,b) \varsubsetneq \mathbb R$. For $\lambda \geq 0$ and $x \in (a,b)$ we set
\begin{equation*}
v_{\lambda,a,b} (x) = \mathbb N_x \left[ 1 - \indic{\mathcal R \subset (a,b)} \, \exp(-\lambda \sigma)\right].
\end{equation*}
The mapping $x \mapsto v_{\lambda,a,b} (x)$ satisfies the ODE
\begin{equation}
\label{eqn:ODE ab local mu infinite}
\begin{cases}
\frac{1}{2} v'' = v^\alpha - \lambda \quad \text{on $(0,+\infty)$},\\
\lim_{x \downarrow a} v(x) = + \infty \text{ if $ a = - \infty$ and $\lambda^{1/\alpha}$ if $a$ is finite} ,\\
\lim_{x \uparrow b} v(x) = + \infty \text{ if $ b = + \infty$ and $\lambda^{1/\alpha}$ if $b$ is finite}.
\end{cases}
\end{equation}
\end{theorem}

Unfortunately, for generic parameters $a$ and $b$, this ODE has no closed form solution. This is also true for the Brownian case, but for $a = 0$ and $b = \infty$ Delmas~\cite[Lemma 7]{delmas2003computation} provides an explicit solution for all $\lambda \geq 0$. This is not possible in the stable case, but we can still derive precise asymptotics when $x$ is close to $0$; see Corollary~\ref{cor:v(x) expansion}. Moreover, we have some explicit expressions in the case $\lambda =0$, leading to:

\begin{theorem}\label{thm:range snake intro}
For any $x \in \R$,
\begin{equation}\label{eq:rangesnake intro}
\mathbb N_x \left[ 0 \in \mathcal R \right] = \left( \frac{\alpha + 1}{(\alpha - 1)^2} \, \frac{1}{x^2} \right)^{\frac{1}{\alpha -1}}  \quad \text{and} \quad \N^{(1)}_0[R^{\frac{2}{\alpha-1}}] = {\Gamma\left(1 - \frac{1}{\alpha}\right)} \left( \frac{\alpha + 1}{(\alpha - 1)^2} \right)^{\frac{1}{\alpha -1}}.
\end{equation}
\end{theorem}

\bigskip

We also study the occupation measure of the snake around its minimal point. Here we are able to compute the expected occupation measure using connections with Bessel processes, extending results by Le Gall and Weill~\cite{legall-weill} in the Brownian case.
The snake occupation measure around its minimum is the random measure $\overline{\I}$ defined by
\begin{equation*}
\langle \overline{\I} , f \rangle = \int_0^\sigma dt \, f(Z_t - \inf_{s\in [0, \sigma]} Z_s). 
\end{equation*}
We establish a first moment formula for the occupation measure under \Ito excursion measure in terms of Bessel processes in Proposition~\ref{prop:momentI}.
This formula implies the following. 

\begin{theorem}\label{cor:I exp vol UB}
For every $\lambda >0$, one has
\[
\limsup_{\varepsilon \to 0} \varepsilon^{-\frac{2 \alpha}{\alpha -1}} \mathbb N_0 \left( (1 - e^{-\lambda \sigma}) \, \overline{\mathcal I} \left( [0,\varepsilon] \right) \right) < \infty.
\]
\end{theorem}
In the Brownian case, this estimate was used by Le Gall~\cite{LeGall2007BrownianMapTopological} to establish bounds on the volume of a ball around a typical point in the Brownian sphere. Unfortunately, our statement cannot be used in the same way in the stable case, since the point of the stable tree realising the minimum of the snake is not a typical point.

\bigskip

Finally, we prove that there are no common increase points for the snake and the underlying \Levy tree. This is also a crucial ingredient in many analyses of the Brownian sphere, in particular the programme to establish its topology \cite{LeGall2007BrownianMapTopological}. We refer the reader to \cite[Lemma 2.2]{LeGall2007BrownianMapTopological} for a proof in the Brownian case.

For $s \in [0,1]$ we say that $s$ is an \textbf{increase point} of the pair $(H,Z)$ if there exists $\epsilon>0$ such that $H_t \geq H_s$ and $Z_t \geq Z_s$ for all $t \in [s, (s+\epsilon) \wedge 1]$. The main result is summarised in the following statement.

\begin{theorem}\label{thm:no common increase points intro}
$\mathbb N_x$-almost everywhere, the pair $(H,Z)$ has no increase points.
\end{theorem}


\textbf{Acknowledgements.} We would like to thank Marie Albenque and Meltem \"Unel for lots of helpful conversations, and Nicolas Curien for pointing out an error in an earlier version of the paper. EA and LM were supported by the ANR grant ProGraM (ANR-19-CE40-0025).

\section{Prerequisites on stable trees and snakes}
\label{section:stable-trees}

\subsection{Stable trees}

\subsubsection{Stable \Levy processes and the \Ito excursion measure}\label{sctn:Levy and Ito}
We consider a spectrally positive stable process $X$ with index $\alpha \in (1,2)$. The law of $X$ started at $0$ is denoted by $P$. This process is characterised by its Laplace transform and the fact that it has no negative jumps; for $\lambda >0$ and $t \ge 0$ we will assume that $X$ is normalised so that
\begin{equation}\label{eqn:Levy process Laplace transform}
\E{ \exp \{- \lambda X_t \}} = \exp \{-t \lambda^\alpha \}.
\end{equation}
(See \cite[Sections VII and VIII]{BertoinLevy} for more background.)

Henceforth we write $\psi(\lambda) = \lambda^{\alpha}$ to denote this Laplace exponent. The function $\psi$ is known as the \textbf{branching mechanism}. By the L\'evy--Khinchin formula, $\psi$ can equivalently be written in the form
\begin{equation}\label{eqn:psi Levy Khinchin}
\psi ( \lambda) = \int_0^{\infty} (e^{-\lambda r} -1 + \lambda r) \pi (dr) = C_{\alpha}\lambda + \int_0^{\infty} (e^{-\lambda r} -1 + \lambda r \indic{y \leq 1}) \pi (dr),
\end{equation}
where $\pi$ is the \textbf{jump measure} of $X$ and $-C_{\alpha}$ is the \textbf{drift coefficient}. The fact that $\psi(\lambda) = \lambda^{\alpha}$ entails that 
\begin{equation}\label{eqn:pi and C def}
C_{\alpha}=\frac{\alpha-1}{\Gamma (2-\alpha)} \quad \text{ and } \quad \pi (dr) = \alpha C_{\alpha} r^{-\alpha-1} dr.
\end{equation}
In the usual terminology, this means that the process $X$ corresponds to the \textbf{\Levy triple} $(0, -C_{\alpha}, \pi)$.

To define a stable tree, we will in fact use a normalised L\'evy excursion rather than a L\'evy process. Let $X$ be an $\alpha$-stable spectrally positive L\'evy process, normalised as in \eqref{eqn:Levy process Laplace transform}, and let $I_t = \inf_{s \in [0,t]} X_s$ denote its running infimum process. Define $g_1$ and $d_1$ by
\begin{align*}
g_1 &= \sup \{ s \leq 1: X_s = I_s \} \\
d_1 &= \inf \{ s > 1: X_s = I_s \}.
\end{align*}

Note that $X_{g_1} = X_{d_1}$ almost surely since $X$ almost surely has no jump at time $g_1$ and $X$ has no negative jumps. We define the normalised excursion $X^{\text{exc}}$ of $X$ above its infimum at time $1$ by
\[
X_s^{\text{exc}} = (d_1 - g_1)^{\frac{-1}{\alpha}} (X_{g_1 + s(d_1 - g_1)} - X_{g_1})
\]
for every $s \in [0,1]$. Then $X^{\text{exc}}$ is almost surely an $\alpha$-stable \cadlag function on $[0,1]$ with $X^{\text{exc}}(s)>0$ for all $s \in (0,1)$, and $X_0^{\text{exc}}=X_1^{\text{exc}}=0$.

We also take this opportunity to introduce the \Ito excursion measure. For full details, see \cite[Chapter IV]{BertoinLevy}, but the measure is defined by applying excursion theory to the process $X - I$, which is strongly Markov and for which the point $0$ is regular for itself. We normalise local time so that $-I$ denotes the local time of $X - I$ at its infimum, and let $(g_j, d_j)_{j \in \mathcal{I}}$ denote the excursion intervals of $X - I$ away from zero. For each $i \in \mathcal{I}$, the process $(e^i)_{0 \leq s \leq d_i-g_i}$ defined by $e^i(s) = X_{g_i + s} - X_{g_i}$ is an element of the excursion space
\[
E = \bigcup_{l > 0} D^{\text{exc}}([0,l], \R^{\geq 0}),
\]
where for every $l >0$, $D^{\text{exc}}([0,l], \R^{\geq 0})$ is the space of all \cadlag nonnegative functions $f$ such that $f(0)=f(l)=0$. We let $\sigma (e) = \sup \{s>0: e(s)>0\}$ denote the \textbf{lifetime} (duration) of the excursion $e$. It was shown in \cite{ItoPP} that the measure on $\mathbb R_+ \times E$
\[
\sum_{i \in \mathcal{I}} \delta (-I_{g_i}, e^i)
\]
is a Poisson point measure of intensity $dt N(de)$, where $N$ is a $\sigma$-finite measure on the set $E$ known as the \textbf{\Ito excursion measure}.

The measure $N(\cdot)$ inherits a scaling property from that of of $X$. Indeed, for any $\lambda > 0$ we define a mapping  $\Phi_{\lambda}: E \rightarrow E$ by  $\Phi_{\lambda}(e)(t) = \lambda^{\frac{1}{\alpha}} e(\frac{t}{\lambda})$, so that $N \circ \Phi_{\lambda}^{-1} = \lambda^{\frac{1}{\alpha}} N$ (e.g. see \cite{WataIto}). It then follows from the results in \cite[Section IV.4]{BertoinLevy} that we can uniquely define a set of conditional measures $(\Ns, s>0)$ on $E$ such that:
\begin{enumerate}[(i)]
\item For every $s > 0$, $\Ns( \sigma=s)=1$.
\item For every $\lambda > 0$ and every $s>0$, $\Phi_{\lambda}(\Ns) = N^{(\lambda s)}$.
\item For every measurable $A \subset E$,
\begin{equation}\label{eqn:Ito measure integrate s}
N(A) = \int_0^{\infty} \frac{\Ns(A)}{\alpha \Gamma(1 - \frac{1}{\alpha}) s^{\frac{1}{\alpha}+1}} ds.
\end{equation}
\end{enumerate}

$\Ns$ is therefore used to denote the law $N( \cdot | \sigma = s)$. The probability distribution $\No$ coincides with the law of $\X$ constructed above.

\subsubsection{Stable height processes and exploration processes}

It is well-known (e.g. see \cite[Figure 1]{LeGDuqMono}) that discrete plane trees can be coded by several different functions. The same is true in the continuum; in the case of $\alpha$-stable trees, we first code using a spectrally positive $\alpha$-stable \Levy process or excursion $X$ (this plays the same role as the \L ukasiewicz path in the discrete setting). The \textbf{height function} $H$ can then be defined by setting it to be the process defined for $t \geq 0$ by
\begin{equation}\label{eqn:height def}
H_t = \lim_{\epsilon \rightarrow 0} \frac{1}{\epsilon} \int_0^t \indic{X_s < I_s^t + \epsilon} ds,
\end{equation}
where $I_{s,t} = \inf_{s \leq r \leq t} X_r$. For each $t \geq 0$, this limit exists in probability by \cite[Lemma 1.2.1]{LeGDuqMono}. Moreover, it follows from \cite[Theorem 1.4.3]{LeGDuqMono} that $H$ almost surely has a continuous modification, and we will assume henceforth that $H$ is indeed continuous.

The height function is not Markovian in general and to overcome this difficulty Le Gall and Le Jan~\cite{le1998branching} introduced the exploration process, a Markovian random measure that encodes the height process. Let $M_f(\R_+)$ denote the set of finite measures on $\R_+$. The \textbf{exploration process} $(\rho_t)_{t\geq 0}$ is a process such that for each $t \geq 0$, $\rho_t$ is the measure in $M_f(\R_+)$ satisfying
\begin{align}\label{def:exploration process}
\langle\rho_t, f\rangle = \int_0^t f(H_s) d_s (I_{s,t})
\end{align}
for all bounded measurable functions $f:\R_+ \to \R$.
Note that this implies that
\begin{align}\label{eqn:rhot 1}
\langle\rho_t, 1\rangle = I_{t,t} - I_{0,t} = X_t - I_t.
\end{align}

Note that it follows from the definition that for each $t \geq 0$, $\text{supp } \rho_t = [0,H_t]$. In particular this means that $H$ can be defined as a functional of $\rho$; for a given realisation $\rho'$ we denote the associated height process by $H(\rho')$. Moreover, by \cite[Proposition 1.2.3]{LeGDuqMono}, the process $(\rho_t)_{t \geq 0}$ is a \cadlag strong Markov process, viewed as a measure taking values in $M_f(\R_+)$ equipped with the topology of weak convergence. In particular this means that the notion of the exploration process started from an initial distribution $\mu$ (provided $\mu$ is a feasible candidate for the exploration process) is well-defined. We denote this law by $\bPb_{\mu} (d \rho)$. This law can also be constructed explicitly using pruning and concatenations as in \cite[Section 2.1]{riera2022structure}. In particular the construction of \cite[Section 2.1]{riera2022structure} makes it clear that this is possible whenever $\supp \mu$ is of the form $[0,a]$ for some $a\in [0, \infty)$, and that $\mu$ is non-atomic.

Finally we introduce the notion of the dual exploration process. Note that it follows from \eqref{def:exploration process} that $\rho_t$ can also be defined by
\begin{align*}
\rho_t(dr) = \sum_{0 < s \leq t: X_{s-} \leq I_{s,t}} (I_{s,t}-X_{s-})  \delta_{H_s}(dr)
\end{align*}
for all $t>0, r \in [0,H_t]$. The \textbf{dual exploration process} is similarly defined by
\begin{align}\label{eqn:dual exp process}
{\hat{\rho}}_t(dr) = \sum_{0 < s \leq t: X_{s-} \leq I_{s,t}} (X_{s} - I_{s,t})  \delta_{H_s}(dr).
\end{align}
The pair $(\rho_t, {\hat{\rho}}_t)_{t \geq 0}$ is also a strong Markov process.

The height function appearing in \eqref{eqn:height def} in fact codes a whole forest of trees. We will sometimes instead code a single tree by replacing $X$ with $\X$ in \eqref{eqn:height def}. This in fact corresponds to the law of $H$ under the measure $\No$. By scaling, this means that we can also make sense of it under $\Ns$ for any $s>0$, and therefore under $N$ using point (iii) of the previous subsection. The pair $(\rho_t, {\hat{\rho}}_t)_{t\geq 0}$ can be similarly defined under $\No$ or $N$. Under $N$, we have that (recall that $\sigma$ denotes the lifetime of the excursion)
\begin{equation} \label{eq:reverserho}
(\rho_t, {\hat{\rho}}_t)_{t\geq 0} \overset{(d)}{=} ({\hat{\rho}}_{(\sigma - t)-}, \rho_{(\sigma - t)-})_{t\geq 0}.
\end{equation}

\subsubsection{Stable trees}\label{sctn:stable tree def}

To define the stable tree under $\No$, we first define $H$ using \eqref{eqn:height def} but with $X$ replaced by $\X$, then for $s<t$ set
\begin{equation}\label{eqn:min st def}
m_{s,t} =  \inf_{s \wedge t \leq r \leq s \vee t} H_r.
\end{equation}
We then define a pseudodistance on $[0,1]$ by
\begin{equation}\label{eqn:distance from height}
\dt(s,t) = H_s + H_t - 2m_{s,t}
\end{equation}
whenever $s \leq t$. We then define an equivalence relation on $[0,1]$ by saying $s \sim t$ if and only if $\dt(s,t) = 0$, and set $\Ta$ to be the quotient space $([0,1]/ \sim, \dt)$. We also define a measure $\lambda$ on $\Ta$ as the image of Lebesgue measure on $[0,1]$ under the quotient operation. 

It follows from the construction that $\Ta$ is an $\R$-tree in the sense of \cite{dress1996t}, essentially meaning that there is a unique path between every pair of points (see \cite[Definition 2.1]{DLG05} for the precise definition).

Properties of $\Ta$ are encoded by the process $\X$. Letting $p: [0, 1] \rightarrow \Ta$ be the canonical projection, we have a distinguished vertex $\rho = p(0)$ which is the root of $\Ta$.

The \textbf{multiplicity} of a vertex $u \in \Ta$ is defined as the number of connected components of $\Ta \setminus \{u\}$. Vertex $u$ is called a \textbf{leaf} if it has multiplicity one, and a \textbf{branch point} if it has multiplicity at least three. It can then be shown that $\Ta$ only has branch points with infinite multiplicity, called \textbf{hubs}, and that $u \in \Ta$ is a hub if and only if there exists $s \in [0, 1]$ such that $p(s)=u$ and $\Delta_s := \X_s - \X_{s^-}>0$. See for instance Theorem 4.6 and Theorem 4.7 of \cite{DLG05}. The quantity $\Delta_s$ gives a measure of the number of children of $u$. 

For fixed $t \geq 0$, the quantity appearing in \eqref{eqn:rhot 1} can be interpreted as the sum of the sizes of the hubs appearing on the right hand side of the branch from the root to $p(t)$. Similarly $\langle{\hat{\rho}}_t, 1\rangle$ can be interpreted as the sum of the sizes of the hubs appearing on the left hand side of the branch from the root to $p(t)$.

We can similarly define $\Ta$ under $N$, rather than $\No$, by replacing $\X$ with an excursion defined under $N$ in the definition of $H$, or similarly define a forest of trees by retaining $X$ in \eqref{eqn:height def} (each excursion of $X-I$ above its infimum then codes a single tree). It is a well-known fact that the genealogy of $\Ta$ encodes a continuous-state branching process with branching mechanism $\psi$ (but we will not use this connection and simply refer to \cite{LeGDuqMono} for further background).

\subsection{Stable snake with Brownian spatial displacements}

\subsubsection{Definition} \label{section:stablesnake}
Here we give a brief introduction to \Levy snakes in the particular case of snakes on stable trees with Brownian spatial displacements. We will often simply refer to this as a ``stable snake" in what follows. A full introduction in the more general setting of \Levy trees is given in \cite[Section 4]{LeGDuqMono}; we refer there for further details and background.

Informally, a stable snake is simply a real-valued stochastic process indexed by a stable tree. Just as we had to introduce the exploration process as a Markovian version of the height process, it will be more convenient to keep track of some of the history of this stochastic process and we therefore in fact consider a path-valued process. In particular, we let $\W$ denote the set of continuous real-valued functions defined on a compact interval of the form $[0, \zeta]$ for some $\zeta>0$. Given $w \in \W$, we denote its domain by $[0, \zeta_w]$ and set $\widehat{w} = w(\zeta_w)$. We call $\zeta_w$ the \textbf{lifetime} of $w$ and $\widehat{w}$ the \textbf{tip} of $w$.

We define a topology on $\W$ using the distance function
\[
d_{\W}(w, w') = |\zeta_w - \zeta_{w'}| + \sup_{r \geq 0} |w(r \wedge \zeta_{w}) - w'(r \wedge \zeta_{w'})|.
\]
Before defining the stable snake, we first define a snake with Brownian spatial displacements driven by a deterministic continuous function $h:\R_+ \to \R_+$ with initial condition $w_0: [0, h(0)] \to \R_+$ as a random variable taking values in $\W$, or in other words in the space of path-valued processes (one should have in mind that $h$ plays the role of a height function of a deterministic forest, and the snake is obtained by adding random spatial displacements). For every $x \in \mathbb R$, we denote by $\Pi_x$ the distribution of a standard linear Brownian motion $\xi = (\xi_t)_{t \geq 0}$ started at $x$. It is shown in \cite[Section 4.1.1]{LeGDuqMono} that for any such choice of $(h, w_0)$ there exists a unique probability measure on $\W$ such that the following three points hold.  
\begin{enumerate}
\item $\zeta_{W_s} = h(s)$ for all $s \in [0, \sigma]$.
\item \label{item:snake-consistency} For all $s, s' \geq 0$, we have that $W_s(t) = W_{s'}(t)$ for all $t \leq m_{s',s}$.
\item Conditionally on $W_{s'}$, the path $W_s:[0, \zeta_s] \to \R$ is such that $(W_s(m_{s',s}+t) - W_s(m_{s',s}))_{t\in [\zeta_s - m_{s',s}]}$ has the same law as the Brownian motion $(\xi_t)_{t\in [0,\zeta_s - m_{s',s}]}$ under $\Pi_{W_{s'}(m_{s',s})}$.
\end{enumerate}

In other words, conditionally on $W_{s'}$, the path $W_s$ satisfies point~\ref{item:snake-consistency} for $t \leq m_{s, s'}$ (which informally corresponds to the snake along the path to the most recent common ancestor of $p(s)$ and $p(s')$), and for $t > m_{s, s'}$ evolves as an independent linear Brownian motion (which corresponds to the rest of the branch to $p(s)$). Given $h$ and $w_0$ as above we let $Q_{w_0}^{h}$ denote the law of this process.

The \textbf{stable snake with Brownian spatial displacements} is the pair $(\rho, W)$ obtained by replacing $h$ with the random height function $H$ defined from a \Levy process $X$ as in \eqref{eqn:height def}. In particular, since $H$ can be written as a function of the exploration process, we consider the pair $(\rho, W)$ with initial condition $(\mu, w)$ to be a random process with law denoted $\Pb_{\mu, w}$ defined by
\begin{align}\label{eqn:law of stable snake def}
\Pb_{\mu, w}(d\rho, dW) = \bPb_{\mu} (d \rho) Q_w^{H(\rho)}(dW).
\end{align}
(Here we assume $\mu \in M_f(\R_+)$ and $w$ clearly needs to be compatible with $\mu$ in the sense that $\supp \mu = [0,\zeta_w]$.) It is shown in \cite[Theorem 4.1.2]{LeGDuqMono} that the process $(\rho_s, W_s)_{s \geq 0}$ is a strong Markov process under the law $\Pb_{\mu, w}$, with respect to $(\F_{s+})_{s \geq 0}$, where $(\F_s)_{s \geq 0}$ is the canonical filtration on the space $\mathbb{D}(\R_+, M_f(\R_+) \times \W)$, the space of \cadlag functions from $\R_+$ to $M_f(\R_+) \times \W$ with respect to the aforementioned topologies and the product topology on the latter space.
We will also use the notation
\begin{equation} \label{eq:tipDef}
\Wt_s = W_s(\zeta_s)
\end{equation}
to denote the tip of the snake at time $s$. 

We will sometimes refer to this construction as simply the ``the snake" or ``the stable snake" in this article. We will also denote the case $\Pb_{0,x}$ by $\Pb_{x}$. In addition, $\Pb^*_{\mu, w}$ will refer to the law of $(\rho, W)$ under $\Pb_{\mu, w}$ but killed when $\rho$ first hits $0$ (in practical terms, this corresponds to starting the process ``partway" through a tree/snake pair and stopping at the moment when the entire tree and snake have been discovered).

\subsubsection{Excursion measures for snakes}

Although the notation $\Pb_{\mu, w}$ always refers to the law of the snake as defined in \eqref{eqn:law of stable snake def}, we will also sometimes need to consider the law of a snake in which $\rho$ is first sampled under $N$ or $\No$ with initial condition consisting of the empty trajectory for which the snake takes initial value $x \in \R$. Informally, under $\No$ this is the law of a stable snake on a single stable tree of total mass $1$, started from the point $x$. Under $N$ this is the law of a stable snake on a single stable tree of total mass sampled according to the It\^o measure, started from the point $x$. We denote these excursion measures for snakes by $\N_x$ and $\Noo_x$, so that
\begin{align}\label{eqn:Ito snake def}
\N_x (d \rho, d W) = N_0 (d \rho) Q_x^{H(\rho)}(dW).
\end{align}
$\Noo_x$ is defined similarly by replacing $N$ with $\No$ above. It is explained in \cite[Section 2.4]{riera2022structure} that the point $(0,x)$ is regular and recurrent for the Markov process $(\rho, W)$, and moreover that the process $(-I_t)_{t\geq 0}$ is a local time at $0$ for this process, and hence this definition does indeed make sense as the excursion measure of $(\rho, W)$ away from $(0,x)$ associated with the local time $-I$.

In other words, by excursion theory, it moreover follows that, if $(\alpha_i, \beta_i)_{i \in \mathcal{I}}$ denote the excursion intervals of $(\rho, W)$ away from $(0,x)$, with corresponding subtrajectories $(\rho^{(i)}, W^{(i)})$, then under $\Pb_x$, the measure
\[
\sum_{i \in \mathcal{I}} \delta (-I_{\alpha_i}, \rho^{(i)}, W^{(i)})
\]
is a Poisson point measure with intensity $\indic{[0, \infty)}(\ell) d \ell \N_x(d \rho, d W)$ (where $d \ell$ denotes Lebesgue measure).

\bigskip

Note that, under $\mathbb N_x$, the function $(\Wt_s)_{s \in [0, \sigma]}$ can also be viewed as a function of the tree $\Ta$ coded by the height function $H(\rho)$ driving the snake that appears in \eqref{eqn:law of stable snake def}. We will follow the convention of \cite{LeGall2007BrownianMapTopological} and denote this function by $(Z_a)_{a \in \Ta}$, so that $Z_a = \Wt_t$ for any $a \in \Ta$ such that $a = p(t)$. We will sometimes abuse notation and use interchangeably $Z_a$, $Z_t$ and $\Wt_t$.
In this setting, we define the \textbf{snake occupation measure} as the random measure $\mathcal I$ satisfying
\begin{equation}\label{eqn:ISE def}
\langle \mathcal I , f \rangle = \int_{\Ta} \lambda(da) \, f(Z_a) = \int_0^\sigma dt \, f(\Wt_t). 
\end{equation}
When $\alpha = 2$, the measure $\mathcal I$ is the so-called Integrated Superbrownian Excursion. The measure $\mathcal I$ is supported on the range of the stable snake, denoted by
\begin{equation} \label{eq:rangeDef}
\mathcal R := \left[ \inf_{a \in \Ta} Z_a, \sup_{a \in \Ta} Z_a \right].
\end{equation}

\bigskip

The measure $\N_x$ also inherits a scaling property from that of of $N(\cdot)$ and in addition the Brownian scaling property. For any $\lambda > 0$ we again define a mapping  $\Phi_{\lambda}: (\rho, W) \to (\rho^{\lambda}, W^{\lambda})$ defined by 
\begin{align}\label{eqn:Ito scaling}
\rho^{\lambda}_t (dr) = \lambda^{\frac{1}{\alpha}} \rho_{t/\lambda} (d(r/\lambda^{1-\frac{1}{\alpha}})), \qquad W_t^{\lambda}(s) = \lambda^{\frac{1}{2}\left(1-\frac{1}{\alpha}\right)} W_{t/\lambda}(s/\lambda^{1-\frac{1}{\alpha}})
\end{align}
so that now the law of $(\rho^{\lambda}, W^{\lambda})$ under $\N_x$ is equal to $\lambda^{1/\alpha}\N_{\lambda^{\frac{1}{2}\left(1-\frac{1}{\alpha}\right)}x}$.

Note also that we can use \eqref{eqn:Ito measure integrate s} and \eqref{eqn:Ito scaling} to calculate that 
\begin{align}\label{eqn:Laplace sigma}
\mathbb N_y \left[ 1 - e^{-\lambda \sigma} \right] = \lambda^{1/ \alpha}
\end{align}
\bigskip

We will use a property of stable snakes known as \textbf{uniform re-rooting invariance}. To define this under $\N_0$ for an excursion with lifetime $\sigma>0$, we follow the presentation of \cite[Section 2.3]{legall-weill} for the Brownian case. As was the case there, we first need to define the notion of a stable tree and of a snake rooted at a time point $s \in [0, \sigma]$. To this end, given $\sigma>0$ first set
\begin{equation*}
s \oplus r = \begin{cases}
s+r &\text{ if } s+r \leq \sigma, \\
s+r - \sigma &\text{ if } s+r > \sigma.
\end{cases}
\end{equation*}
We then define the re-rooted lifetime and snake-tip processes by
\begin{align}\label{eqn:rerooted height and snake tip}
\zetas_r &= \zeta_s + \zeta_{s \oplus r} - 2\inf_{u \in [s \wedge (s \oplus r), s \vee (s \oplus r)]} \zeta_u \qquad \text{ and } \qquad \Wts_r = \Wt_{s \oplus r} - \Wt_s
\end{align}
for all $r \in [0, \sigma]$. Note that the entire sequence of snake trajectories can be recovered directly from \eqref{eqn:rerooted height and snake tip} via
\begin{align*}
\Ws_r (t) = \Wts_{\sup\{u \leq r: \zetas_u=t\}},
\end{align*}
so \eqref{eqn:rerooted height and snake tip} does indeed suffice to define the process $((\Ws_r (t))_{t \in [0,\zetas_r]})_{ r \in [0,\sigma]}$.

Informally, $(\zetas_r)_{r \geq 0}$ codes a stable tree rooted at the vertex previously labelled $s$, and $\Wts_r$ codes the spatial displacements along the branches of the new tree from the new root to the vertex which now has label $r$.

The re-rooting invariance statement is as follows.

\begin{proposition}\label{prop:uniform rerooting}
For every non-negative function $F$ on $\R_+ \times C(\R_+, \W)$ it holds that
\[
\N_0 \left( \int_0^{\sigma} F(s, \Ws) ds \right) = \N_0 \left( \int_0^{\sigma} F(s, W) ds \right).
\]
\end{proposition}
\begin{proof}
This was proved in \cite[Theorem 2.3]{legall-weill} via two steps:
\begin{enumerate}
\item Firstly showing that the process $(\zeta_s)_{s \in [0, \sigma]}$ (coding the underlying tree) satisfies uniform re-rooting invariance; in the stable case this has been shown \cite[Proposition 4.8]{DLG05}.
\item Then observing the Gaussian increments and covariance structure of the snake are preserved under re-rooting in order to extend the invariance of $(\zeta_s)_{s \in [0, \sigma]}$ to joint invariance of the pair $(\zeta_s, W_s)_{s \in [0, \sigma]}$.
\end{enumerate}
The proof in the stable case therefore follows by exactly the same logic and we refer to \cite[Theorem 2.3]{legall-weill} for full details.
\end{proof}

\subsubsection{Spinal decompositions}\label{sctn:spinal decomp}
In this section we present an important spinal decomposition formula, which will allow us to express certain integral functionals of the stable snake under $\N_0$ using a decomposition over subtrees and subtrajectories along a branch to a uniform point in the underlying tree.

We start by giving a Poisson process decomposition of subtrees from a deterministic initial condition. Recall that under $\mathbb P_{\mu,w}^*$, the process $Y_t = \langle\rho_t,1\rangle$ is distributed as the underlying stable process $X_t - I_t$ started at $\langle\mu,1\rangle$ and stopped when it first hits $0$. Also recall that the height process can be recovered by $\mathrm{supp} \, \rho_t = [0,H_t]$. Set $K_t = \inf_{s \leq t} Y_s$ and denote by $(\alpha_i,\beta_i)_{i \in I}$ the excursion intervals of $Y_t - K_t$ away from $0$. For every $i \in I$, we set $h_i = H_{\alpha_i} = H_{\beta_i}$ and define the pair $(\rho^i,W^i)$ by
\begin{equation}
\begin{cases}
\langle \rho^i_s , f \rangle = \int_{h_i}^\infty \rho_{\alpha_i + s}(dr) \, f(r - h_i) & \text{for $0 \leq s \leq \beta_i - \alpha_i$,}\\
\rho^i_s = 0 & \text{for $ s > \beta_i - \alpha_i$,}
\end{cases}
\end{equation}
and
\begin{equation}
\begin{cases}
W^i_s (t) = W_{\alpha_i + s}(h_i + t), \, \zeta^i_s = H_{\alpha_i + s} - h_i & \text{for $0 < s < \beta_i - \alpha_i$,}\\
W^i_s = w(h_i) & \text{for $ s= 0 $ and $ s \geq \beta_i - \alpha_i$}.
\end{cases}
\end{equation}
We have the following Poisson decomposition for the sub-excursions.

\begin{lemma}[Lemma 4.2.4 of \cite{LeGDuqMono}] \label{lemma:onesidePoisson}
Under $\mathbb P^*_{\mu,w}$, the point measure
\[
\sum_{i \in I} \delta_{(h_i,\rho^i,W^i)}
\]
is a Poisson point process with intensity
\[
\mu(dh) \, \mathbb N_{w(h)} \left( d\rho \, dW\right). 
\]
\end{lemma}

Note that the above decomposition is one-sided in the sense that it only codes the subtrees on the right hand side of the initial branch. We will also need a two-sided version for the spine to a uniform point. The final formula (\cref{prop:firstmomentrange ds}) combines a few ingredients:
\begin{enumerate}
\item A so-called many-to-one formula which allows us to express integral functions of the snake as the expectation over two subordinators which give the law of the exploration process and its dual projected along a branch to a uniform time point.
\item A Poisson-process decomposition which determines, conditionally on the values of the exploration process and its dual projected along a branch to a uniform time point, the laws of the subtrees (and associated snakes) attached to this branch, and their locations on it.
\item Combining versions of the above two points existing in the literature gives us a special case of the desired formula, when the relevant functionals are in some sense local. We then prove an extension to functionals which can in particular depend on the whole range of the snake.
\end{enumerate}

For the first of these, we let $(U^{(1)}, U^{(2)})$ be a two-dimensional subordinator defined on some probability space $(\Omega_0,\mathcal F_0, P^0)$, started at $(0,0)$ and with Laplace transform (recall that $\psi(\lambda) = \lambda^{\alpha}$)

\begin{align}\label{eqn:U Laplace transform}
E^0[\exp\{-\lambda_1 U^{(1)}_t - \lambda_2 U^{(2)}_t\}] = \begin{cases}
\exp\left\{-\frac{t (\psi(\lambda_1)-\psi(\lambda_2))}{\lambda_1 - \lambda_2}\right\} &\text{ if } \lambda_1 \neq \lambda_2, \\
\exp\{-t \psi'(\lambda_1)\} &\text{ if } \lambda_1 = \lambda_2.
\end{cases}
\end{align}
Note that the marginal of each of $U^{(1)}$ and $U^{(2)}$ is that of a subordinator with Laplace exponent $\tilde{\psi}(\lambda) := \frac{\psi(\lambda)}{\lambda} = \lambda^{\alpha -1}$, and the sum $U^{(1)} + U^{(2)}$ is a subordinator with Laplace exponent $\psi'(\lambda) = \alpha \, \lambda^{\alpha -1}$. Similarly to \eqref{eqn:psi Levy Khinchin}, this entails that
\begin{align}\label{eqn:psi tilde prime Laplace rep}
\begin{split}
\tilde{\psi} ( \lambda) &= \int_0^{\infty} (1-e^{-\lambda r} ) \tilde{\pi} (dr) = \frac{C_{\alpha}\lambda}{2-\alpha} + \int_0^{\infty} (1-e^{-\lambda r} - \lambda r\indic{\{r \leq 1\}}) \tilde{\pi}(dr), \\
\psi' ( \lambda) &= \int_0^{\infty} (1-e^{-\lambda r} ) \pi' (dr) = \frac{\alpha C_{\alpha}\lambda}{2-\alpha} + \int_0^{\infty} (1-e^{-\lambda r}- \lambda r\indic{\{r \leq 1\}}) \pi' (dr).
\end{split}
\end{align}
where $\tilde{\pi} (dr) = C_{\alpha} r^{-\alpha} dr$ and $\pi' (dr) = \alpha C_{\alpha} r^{-\alpha} dr$ (these can be calculated directly from \eqref{eqn:psi Levy Khinchin} and \eqref{eqn:pi and C def} but see \cite[Section III.1]{BertoinLevy} for further background on stable subordinators). We note only that $\tilde{\psi}$ can be easily recovered from $\psi'$ by noting that $U^{(1)}$ is obtained from $U^{(1)} + U^{(2)}$ through the fact that a jump at time $t$ of $U^{(1)}$ corresponds to a uniform portion of the corresponding jump at time $t$ of $U^{(1)} + U^{(2)}$, hence (using Fubini's theorem)
\begin{align}\label{eqn:subordinator uniform split}
\tilde{\psi} ( \lambda) =  \int_0^{\infty} (1-e^{-\lambda \ell} ) \int_0^r \frac{1}{r} d \ell \pi' (dr) = \int_0^{\infty} (1-e^{-\lambda \ell} ) \int_{\ell}^{\infty} \frac{1}{r} \pi' (dr)  d \ell = \int_0^{\infty} (1-e^{-\lambda \ell} ) \tilde{\pi} (d\ell).
\end{align}

We will continue to use the notation $\psi, \psi'$ and $\tilde{\psi}$ throughout the paper to refer to these three functions above.

For each $a>0$, let $(J_a, \J_a)$ denote a pair of random measures given by
\begin{align} \label{eq:defsJ}
(J_a, \J_a) (dt) = (\indic{[0,a]}(t) dU_t^{(1)}, \indic{[0,a]}(t) dU_t^{(2)})
\end{align}

The many-to-one formula is as follows.

\begin{proposition}[Lemma 1 in \cite{riera2022structure}; cf also Lemma 3.4 in \cite{DLG05}] \label{prop:firstmoment ds}
For every $x \in \R$ and every non-negative measurable functional $\phi$ taking values in $M_f(\R_+)^2 \times \W$, we have
\begin{align*}
\N_x \left[ \int_0^{\sigma} \phi (\rho_s, {\hat{\rho}}_s, W_s) ds \right] = \int_0^\infty dh \, E^0 \otimes \Pi_x \left( \phi(J_h, \J_h, (\xi_s)_{0\leq s \leq h}) \right).
\end{align*}
\end{proposition}

We will need an extension of this result when the functional $\phi$ also depends on the range and lifetime of the snake (see Theorem 2.2 of~\cite{legall-weill}), and moreover where we also consider the laws of all the subtrees grafted to the spine from the root to $p(s)$. For $h > 0$, under $E^0 \otimes \Pi_x$, conditionally on the measure $(J_h + \hat J_h)$ and the Brownian trajectory $(\xi_s)_{0\leq s \leq h}$, we consider an independent Poisson point process on $[0,h] \times \mathcal W$, denoted
\[
\mathcal P := \sum_{i\in I} \delta_{a_i,w_i},
\]
with intensity
\begin{equation}
(J_h + \hat J_h)(da) \otimes \mathbb N_{\xi_a} (dw).
\end{equation}
The extension we need is as follows. In the proposition we let $\mathcal{R}$ denote the \textit{range} of the snake $W$, and we let $\mathbb{E}_{\Pcal}$ denote expectation with respect to the point process $\Pcal$ defined above.

\begin{proposition}[cf Theorem 2.2 in \cite{legall-weill}] \label{prop:firstmomentrange ds}
For every $x \in \R$ and every non-negative measurable functional $\phi$ taking values in $M_f(\R_+)^2 \times \W$, we have
\begin{align*}
&\N_x \left[ \int_0^{\sigma} \phi (\rho_s, {\hat{\rho}}_s, W_s,\mathcal R, \sigma) ds \right] \\
&\qquad = \int_0^\infty dh \, E^0 \otimes \Pi_x \left [
\mathbb E_{\mathcal{P}} \left[
\phi \left( J_h, \J_h, (\xi_s)_{0\leq s \leq h}, \overline{\bigcup_{i \in I} \left(\zeta_{a_i} + \mathcal R(w_i) \right)} , \sum_{i \in I} \sigma (w_i)\right)
\middle| (J, \J, \xi)
\right] \right].
\end{align*}
\end{proposition}
\begin{proof}
The statement without the range follows from combining \cref{prop:firstmoment ds} with \cite[Theorem 4.5]{DLG05} (additionally integrating over $a$ in the statement of \cite[Theorem 4.5]{DLG05}). In particular, repeat the proof of \cite[Theorem 4.5]{DLG05}, and use \cref{prop:firstmoment ds} as an input in place of \cite[Lemma 3.4]{DLG05}, along with \cref{lemma:onesidePoisson}. Then integrate over $a$ in the final statement. The stated result then follows from \cref{prop:firstmoment ds} using the same proof as Theorem 2.2 in \cite{legall-weill}, which deduces the Brownian version of this proposition from the Brownian analogue without the range.
\end{proof}

In order to carry out calculations for Poisson processes such as those appearing in \cref{lemma:onesidePoisson} and \cref{prop:firstmomentrange ds}, and with respect to the measures $J$ and $\hat{J}$, we remind the reader of two standard formulas.

Firstly note that, using the Markov property and stationarity of increments, the formula \eqref{eqn:U Laplace transform} easily generalises to the following for continuous nonnegative functions $f$.
\begin{align}\label{eqn:U Laplace transform cts}
\begin{split}
E^0\left[\exp\left\{-\int_0^t f(s) d(U^{(1)} + U^{(2)})_s \right\}\right] &= 
\exp\left\{-\int_0^t \psi'(f(s)) ds \right\} \\
E^0\left[\exp\left\{-\int_0^t f(s) dU^{(1)}_s \right\}\right] &= 
\exp\left\{-\int_0^t \tilde{\psi}(f(s)) ds \right\}.
\end{split}
\end{align}

Secondly, we will use the following formula for Poisson point processes several times.

\begin{proposition}\label{prop:Poisson master formula}
Let $P = \sum_{i \in \mathcal{I}} \delta (\omega_i)$ denote a Poisson point process on a space $\Omega$ with intensity measure $\hat{\mu}$ and let $f$ be a function $\Omega \to \R$. Then
\[
\E{\prod_{i \in \mathcal{I}} f (\omega_i)} = \exp \left\{\int_{\Omega} (f (\omega)-1) \hat{\mu}(d \omega)\right\}.
\]
\end{proposition}

\subsubsection{Exit measures and the special Markov property}\label{sctn:exit measures and SMP}

In this section we consider the laws of subtrees and associated snakes that branch out from certain level sets of the snake (specifically the times when the snake first exits some domain $D$). In order to formalise this, we will need some notation, following the presentation of \cite[Section 2.4]{legall-weill} and \cite[Section 3]{riera2022structure} for the stable case. We also refer to \cite[Section 4.3]{LeGDuqMono} for some further background on exit measures.

We fix some $x \in \R$ and let $D$ denote a domain in $\R$ such that $x \in D$. For $w \in \W$ we define 
\[
\tau^D(w) = \inf \{t \geq 0: w(t) \notin D\}
\]
(where we interpret the infimum of the empty set as infinity). We now consider the pair $(\rho, W)$ sampled under $\N_x$. Note that the random set $\{s \geq 0 : \zeta_s > \tau^D(W_s)\}$ is open $\N_x$-almost everywhere and can therefore ($\N_x$-almost everywhere) be written as a disjoint union of open intervals of the form $(a_i, b_i)_{i \in I}$ satisfying, for all $s \in (a_i, b_i)$,
\[
\tau^D(W_{s}) = \tau^D(W_{a_i}) = \tau^D(W_{b_i}) = \zeta_{a_i} = \zeta_{b_i}.
\]
Moreover for all $s \in [a_i, b_i]$, the paths $W_s$ are identical up until their first exit time from $D$.

We first define a process that keeps track of the evolution of snake trajectories before their exit times from $D$; this is defined as a time-change of the original snake by setting 
\begin{equation} \label{eq:exit-theta}
\theta^D_s = \inf \left\{ r \geq 0: \int_0^r \indic{\tau^D(W_u) > \zeta_u} du > s \right\} \qquad \text{ and } \qquad \widetilde{W}^D_s = W_{\theta^D_s}.
\end{equation}
We then let $\mathcal{E}^D$ denote the $\sigma$-algebra generated by $\widetilde{W}^D$, completed with the set of $\N_x$-negligible sets.

Our aim is to consider the laws of the ``excursions" of the snake during the time intervals $(a_i, b_i)_{i \in I}$. We define the collection of excursions outside of $D$ as $(\rho^{(i)},W^{(i)})_{i \in I}$ taking values in $M_f(\R_+) \times \W$ such that for all $s \in [a_i, b_i]$ (for all bounded measurable functions $f:\R_+ \to \R$),
\begin{align*}
\langle \rho^{(i)}_s, f \rangle &= \int f(h-\tau^D(W_{a_i})) \indic{h>\tau^D(W_{a_i})}  \rho_{a_i+s}(dh) \\
W^{(i)}_s(t) &= W_{(a_i+s) \wedge b}(t+\zeta_{a_i}), \qquad 0 \leq t \leq \zeta_{(a_i+s) \wedge b} - \zeta_{a_i}.
\end{align*}

In order to properly index these excursions, we also need the notion of \textbf{exit local time}, defined for each $s \geq 0$ via the approximation
\begin{equation}\label{eqn:exit measure local time def}
L_s^D = \lim_{\epsilon \downarrow 0} \frac{1}{\epsilon} \int_0^s \indic{ \tau^D(W_r) < H_r < \tau^D(W_r) + \epsilon }dr
\end{equation}
(indeed it follows from \cite[Propositions 4.3.1 and 4.3.2]{LeGDuqMono} that the limit exists for all $s \geq 0$). By \cite[Proposition 3]{riera2022structure}, the time-changed process $(L_{\theta_s^D})_{s \geq 0}$ is measurable with respect to $\mathcal{E}^D$, and the measure $d L^D_s$ is supported on the set $\{s\geq 0: \Wts \in \partial D\}$. The \textbf{exit measure from $D$} is the measure $\Z^D$ satisfying, for all bounded measurable $f: \partial D \to \R_+$,
\[
\langle \Z^D, f \rangle = \int_0^{\sigma} f(\Wt_s) dL_s^D. 
\]

Note that the total mass of $\Z^D$ is $L^D_{\sigma}$. The main proposition is as follows.

\begin{proposition}[Corollary 1 in~\cite{riera2022structure}]\label{prop:special MP}
Set $T_D = \inf \{ t \geq 0: \tau^D(W_t) < \infty\}$ (again interpreting the infimum of the empty set as infinity). Then, under $\N_x(\cdot | T_D < \infty)$ and conditionally on $\mathcal{E}^D$, the point measure
\[
\sum_{i \in I} \delta(\rho^{(i)}, W^{(i)}) (d \rho, d \omega)
\]
is a Poisson random measure with intensity $\int \N_y(d\rho, d\omega) \Z^D(dy)$.
\end{proposition}

We will also use the following first moment formula (Proposition 4.3.2 of \cite{LeGDuqMono} or Equation (3.4) of \cite{riera2022structure}), which slightly generalises \cref{prop:firstmoment ds}.

\begin{proposition} \label{prop:firstmomentdLs}
For every domain $D$, every $x\in D$, and every non-negative measurable functional $\phi$ taking values in $M_f(\R_+)^2 \times \W$, we have
\begin{align*}
\N_x \left[ \int_0^{\sigma} d L_s^D \, \phi (\rho_s, {\hat{\rho}}_s, W_s) \right] =  E^0 \otimes \Pi_x \left( \indic{\tau^D<\infty}  \phi(J_{\tau_D}, \J_{\tau_D}, (\xi_s)_{0\leq s \leq \tau_D}) \right).
\end{align*}
\end{proposition}

\subsubsection{Re-rooting at the minimum}

Finally, we note the following important consequence of the special Markov property.

\begin{proposition}\label{prop:minimum unique}
$\N_0$-almost everywhere, the time $s^* \in [0,\sigma]$ at which $\Wt$ attains its minimum is unique.
\end{proposition}
\begin{proof}
This follows by exactly the same argument given in \cite[Proposition 2.5]{legall-weill} in the Brownian case.
\end{proof}

In particular, the previous proposition allows us to define the law of the pair $(\rho,W)$ and the associated trees and snakes re-rooted at the minimum under $\N_0$ or $\Noo_0$ as follows. We first sample $(\rho,W)$, let $s^*$ be as in \cref{prop:minimum unique}, and then consider the processes $(\zetass_r)_{r \in [0, \sigma]}$ and $(\Wtss_r)_{r \in [0, \sigma]}$ as defined in \eqref{eqn:rerooted height and snake tip}. Note that, by \cref{prop:minimum unique}, $\Wtss_0 = \Wtss_{\sigma} = 0$, and $\Wtss_r > 0$ for all $r \in (0, \sigma)$, $\N$-almost everywhere.
 
We then let $\overline{\Ta}$ denote the tree coded by the height function $\zetass$ (i.e. using \eqref{eqn:distance from height}), and let $(\overline{Z}_a)_{a \in \overline{\Ta}}$ denote the corresponding snake. We also generalise the definition of $\I$ from \eqref{eqn:ISE def} by setting
\begin{equation}\label{eqn:ISE def rerooted}
\langle \overline{\mathcal I} , f \rangle = \int_{\overline{\Ta}} \lambda(da) \, f(\overline Z_a).
\end{equation}

\section{Snake range and lifetime}\label{sctn:snake range and lifetime}
In this section we fix a domain $D=(a,b)$ (we allow $a$ or $b$ to be infinite). We will use several quantities defined in \cref{sctn:exit measures and SMP} in terms of $D$, but for ease of notation we will suppress the dependence on $D$ (which will always be $(a,b)$).

The next two lemmas give the stable analogues of \cite[Lemma 6]{delmas2003computation} and \cite[Lemma 7]{delmas2003computation}. Note that we work under a slightly different normalisation convention (we take $\psi (\lambda) = \lambda^{\alpha}$ whereas Delmas takes $\psi (\lambda) = 2\lambda^{2}$, meaning that some constants are different).

\begin{lemma} \label{lemma:vlambdamu}
Let $L_\sigma$ denote the total exit local time from $(a,b)$ of the stable snake, i.e. $L_\sigma^{(a,b)}$ in the notation of \eqref{eqn:exit measure local time def}. For $\lambda,\mu \geq0$ and $x \in (a,b)$ we set
\begin{equation*}
v_{\lambda,\mu,a,b} (x) = \mathbb N_x \left[ 1 - \exp(-\lambda \sigma - \mu L_\sigma)\right].
\end{equation*}
The mapping $x \mapsto v_{\lambda,\mu,a,b} (x)$ satisfies the ODE
\begin{equation}\label{eqn:ODE ab local}
\begin{cases}
\frac{1}{2} v'' = v^\alpha - \lambda \quad \text{on $(0,+\infty)$},\\
\lim_{x \downarrow a} v(x) = \mu \indic{a>-\infty} + \lambda^{1/\alpha},\\
\lim_{x \uparrow b} v(x) = \mu \indic{b<+\infty} +  \lambda^{1/\alpha}.
\end{cases}
\end{equation}
\end{lemma}
\begin{proof}
We follow the proof of the Brownian case given in \cite[Lemma 6]{delmas2003computation}.
Suppose $\lambda$ and $\mu >0$.
As in \cref{sctn:exit measures and SMP}, for any $w \in \mathcal W$, we denote by $\tau(w)$ the first exit time from $(a,b)$ of the path $w$ (we set $\tau(w) = +\infty$ if $w$ does not exit $(a,b)$). Recall also from \cref{sctn:exit measures and SMP} the definition of the time change $(\theta_s)_s$ , the exit local times $(L_s)_s$ and of the snake $\widetilde W$ obtained from the trajectories of $W$ before they exit $(a,b)$. From the definition of $\theta$, one can see that the quantity
\[
\int_0^\sigma \indic{\tau(W_s) = +\infty } ds
\]
is $\mathbb N_x$-a.s. equal to the lifetime of the snake $\widetilde W$, and is therefore measurable with respect to the $\sigma$-algebra $\mathcal E$ (recall this is the canonical $\sigma$-algebra for $\widetilde W$).
Therefore, for $\lambda,\mu,x >0$ we have:
\begin{align*}
v_{\lambda,\mu,a,b} (x) &=
\mathbb N_x \left[ 1 - \exp \left(-\lambda \int_0^\sigma \indic{\tau(W_s) = +\infty } ds -\lambda \int_0^\sigma \indic{ \tau(W_s) < +\infty } ds - \mu L_\sigma \right) \right], \\
&=
\mathbb N_x \left[ 1 - \exp \left(-\lambda \int_0^\sigma \indic{\tau(W_s) = +\infty } ds  - \mu L_\sigma \right) 
\mathbb N_x \left[ \exp \left( -\lambda \int_0^\sigma \indic{ \tau(W_s) < +\infty} ds \right) \middle| \mathcal E \right]
\right].
\end{align*}
Denoting by $\mathcal Z$ the exit measure from $(a,b)$ of the stable snake, \cref{prop:Poisson master formula} and \cref{prop:special MP} (the special Markov property) imply that
\begin{align*}
\mathbb N_x \left[ \exp \left( -\lambda \int_0^\sigma \indic{ \tau(W_s) < +\infty} ds \right) \middle| \mathcal E \right]
 = \exp \left( - \int \mathcal Z(dy) \, \mathbb N_y \left[ 1 - e^{-\lambda \sigma} \right] \right)
& = \exp \left( - \lambda^{1/ \alpha} L_\sigma \right),
\end{align*}
where we used \eqref{eqn:Laplace sigma}.

We now define an additive functional of the stable snake $(\widetilde L_s)_{s \geq 0}$ by setting
\begin{equation}\label{eqn:additive functional def}
d \widetilde L_s = \lambda \indic{\tau(W_s) = +\infty } ds  + \left( \lambda^{1/\alpha} + \mu \right) dL_s.
\end{equation}
By the previous calculations, we have that (note that the latter expression can be integrated directly)
\begin{align*}
v_{\lambda,\mu,a,b} (x) &=
\mathbb N_x \left[ 1 - \exp \left(- \widetilde L_\sigma \right) \right] =
\mathbb N_x \left[
\int_0^\sigma d\widetilde L_s \exp \left(- \int_s^\sigma d \widetilde L_u \right)
\right].
\end{align*}
Now recall from Section \ref{section:stablesnake} that $\Pb^*_{\mu, w}$ denotes the law of $(\rho, W)$ under $\Pb_{\mu, w}$ killed when $\rho$ first hits $0$. Using linearity of expectation and the Markov property, the previous identity can be rewritten as
\[
v_{\lambda,\mu,a,b} (x) =
\mathbb N_x \left[
\int_0^\sigma d\widetilde L_s \,  \mathbb E_{\rho_s,W_s}^* \left [ \exp \{ - \widetilde L_\sigma \}\right]
\right].
\]
The Poisson point process decomposition of Lemma \ref{lemma:onesidePoisson} along with \cref{prop:Poisson master formula} then yields
\begin{align*}
 \mathbb E_{\rho_s,W_s}^* \left [ \exp \{ - \widetilde L_\sigma \} \right] &= \exp \left( - \int_0^\infty \rho_s(dh) \mathbb N_{W_s(h)} \left[1 - e^{-\widetilde L_\sigma} \right]\right),\\
& = \exp \left( - \int_0^\infty \rho_s(dh) \, v_{\lambda,\mu,a,b}(W_s(h)) \right).
\end{align*}
Substituting back into the previous display and then using \eqref{eqn:additive functional def}, we now have
\begin{align*}
v_{\lambda,\mu,a,b} (x) &=
\mathbb N_x \left[
\int_0^\sigma d\widetilde L_s \, 
\exp \left( - \int_0^\infty \rho_s(dh) \, v_{\lambda,\mu,a,b}(W_s(h)) \right)
\right],\\
&=
\left( \lambda^{1/\alpha} + \mu \right) \, \mathbb N_x \left[
\int_0^\sigma d L_s \, 
\exp \left( - \int_0^\infty \rho_s(dh) \, v_{\lambda,\mu,a,b}(W_s(h)) \right)
\right] \\
& \quad \quad \quad +
\lambda \, \mathbb N_x \left[
\int_0^\sigma ds \, \indic{\tau(W_s) = + \infty}
\exp \left( - \int_0^\infty \rho_s(dh) \, v_{\lambda,\mu,a,b}(W_s(h)) \right)
\right].
\end{align*}
Letting $J$ be as in \cref{sctn:spinal decomp}, we can express the two expectations with respect to $\mathbb N_x$ involved in the last display using Proposition \ref{prop:firstmomentdLs} and Proposition \ref{prop:firstmoment ds} respectively and therefore write:
\begin{align*}
v_{\lambda,\mu,a,b} (x) &=
\left( \lambda^{1/\alpha} + \mu \right) \, 
E^0 \otimes \Pi_x \left[ \indic{\tau < +\infty} \exp \left( - \int_0^\infty J_{\tau}(dh) v_{\lambda,\mu,a,b} (\xi_h) \right) \right] \\
& \quad \quad \quad +
\lambda \, \int_0^\infty dh \, E^0 \otimes \Pi_x \left[ \indic{\tau > h} \exp \left( - \int_0^\infty J_h(dy) \, v_{\lambda,\mu,a,b}(\xi_y) \right) \right].
\end{align*}
Notice that in the first term, $\Pi_x$-a.s., one has $\tau < +\infty$.
In addition, the two expectations with respect to the random measures $J$ can be expressed with the help of the Laplace transform of their underlying subordinators (recall \eqref{eqn:U Laplace transform} and \eqref{eqn:U Laplace transform cts}) that have exponent $\widetilde \psi (\lambda) = \lambda^{\alpha -1}$:
\begin{align} \label{eq:vexprint}
\begin{split}
v_{\lambda,\mu,a,b} (x)  =
&\left( \lambda^{1/\alpha} + \mu \right) \, 
\Pi_x \left[ \exp \left( - \int_0^\tau \widetilde \psi \left( v_{\lambda,\mu,a,b} (\xi_h) \right) \, dh \right) \right] \\
&\qquad +
\lambda \, \Pi_x \left[ \int_0^\tau dh \, \exp \left( - \int_0^h dy \, \widetilde \psi \left( v_{\lambda,\mu,a,b}(\xi_y) \right) \right) \right].
\end{split}
\end{align}

The ODE of the proposition can now be obtained using standard Feynman-Kac calculus and properties of Brownian motion. We reproduce these arguments here for the sake of completeness. Firstly, one can write
\begin{align*}
v_{\lambda,\mu,a,b} (x) &=
\left( \lambda^{1/\alpha} + \mu \right)
 - \left( \lambda^{1/\alpha} + \mu \right) \, 
\Pi_x \left[ 
\int_0^\tau dh \, \widetilde \psi \left( v_{\lambda,\mu,a,b} (\xi_h) \right)
\exp \left( - \int_h^\tau \widetilde \psi \left( v_{\lambda,\mu,a,b} (\xi_y) \right) \, dy \right) \right] \\
& \qquad +
\lambda \, \Pi_x \left[ \int_0^\tau dh \, \exp \left( - \int_0^h dy \, \widetilde \psi \left( v_{\lambda,\mu,a,b}(\xi_y) \right) \right) \right], \\
&=
\left( \lambda^{1/\alpha} + \mu \right)
 -  
\Pi_x \left[ 
\int_0^\tau dh \, \widetilde \psi \left( v_{\lambda,\mu,a,b} (\xi_h) \right) \left( \lambda^{1/\alpha} + \mu \right) \,
\Pi_{\xi_h} \left[ 
\exp \left( - \int_0^\tau \widetilde \psi \left( v_{\lambda,\mu,a,b} (\xi_y) \right) \, dy \right) \right] \right] \\
& \qquad +
\lambda \, \Pi_x \left[ \int_0^\tau dh \, \exp \left( - \int_0^h dy \, \widetilde \psi \left( v_{\lambda,\mu,a,b}(\xi_y) \right) \right) \right],
\end{align*}
where we only used the Markov property of Brownian motion and linearity of expectation to get the second identity. Using Equation \eqref{eq:vexprint} with $x=\xi_h$ in the last display to substitute for $\left( \lambda^{1/\alpha} + \mu \right)
\Pi_{\xi_h} \left[ 
\exp \left( - \int_0^\tau \widetilde \psi \left( v_{\lambda,\mu,a,b} (\xi_y) \right) \, dy \right) \right]$ in the first step, and then re-indexing time in the second step, we get
\begin{align*}
v_{\lambda,\mu,a,b} (x) &=
\left( \lambda^{1/\alpha} + \mu \right)
  - \Pi_x \left[ 
\int_0^\tau dh \, \widetilde \psi \left( v_{\lambda,\mu,a,b} (\xi_h) \right) \cdot v_{\lambda,\mu,a,b} (\xi_h)
 \right]\\
& \qquad \qquad +
\lambda \, \Pi_x \left[ \int_0^\tau dh \, \widetilde \psi \left( v_{\lambda,\mu,a,b} (\xi_h) \right) \, 
\Pi_{\xi_h} \left[
\int_0^\tau db
\, \exp \left( - \int_0^b da \, \widetilde \psi \left( v_{\lambda,\mu,a,b}(\xi_a) \right) \right) \right] \right]\\
& \qquad \qquad +
\lambda \, \Pi_x \left[ \int_0^\tau dh \, \exp \left( - \int_0^h da \, \widetilde \psi \left( v_{\lambda,\mu,a,b}(\xi_a) \right) \right) \right], \\
&=
\left( \lambda^{1/\alpha} + \mu \right)
  - \Pi_x \left[ 
\int_0^\tau dh \, \psi \left( v_{\lambda,\mu,a,b} (\xi_h) \right)
 \right]\\
& \qquad \qquad +
\lambda \, \Pi_x \left[ \int_0^\tau dh \, \widetilde \psi \left( v_{\lambda,\mu,a,b} (\xi_h) \right) \, 
\int_h^\tau dz
\, \exp \left( - \int_h^z dy \, \widetilde \psi \left( v_{\lambda,\mu,a,b}(\xi_y) \right) \right) \right]\\
& \qquad \qquad +
\lambda \, \Pi_x \left[ \int_0^\tau dh \, \exp \left( - \int_0^h dy \, \widetilde \psi \left( v_{\lambda,\mu,a,b}(\xi_y) \right) \right) \right].
\end{align*}
From there, a straightforward application of Fubini's theorem to the second integral gives
\begin{align*}
 &\Pi_x \left[ \int_0^\tau dh \, \widetilde \psi \left( v_{\lambda,\mu,a,b} (\xi_h) \right) \, 
\int_h^\tau dz
\, \exp \left( - \int_h^z dy \, \widetilde \psi \left( v_{\lambda,\mu,a,b}(\xi_y) \right) \right) \right] \\
&\qquad \qquad =  \Pi_x \left[ \int_0^\tau dz \, \int_0^z dh \, \widetilde \psi \left( v_{\lambda,\mu,a,b} (\xi_h) \right)
\, \exp \left( - \int_h^z dy \, \widetilde \psi \left( v_{\lambda,\mu,a,b}(\xi_y) \right) \right) \right] \\
&\qquad \qquad =  \Pi_x \left[ \int_0^\tau dz \, 
\,\left(1 - \exp \left( - \int_0^z dy \, \widetilde \psi \left( v_{\lambda,\mu,a,b}(\xi_y) \right) \right)\right)\right].
\end{align*}
Substituting back into the previous display we deduce that
\begin{align*}
v_{\lambda,\mu,a,b} (x) &= \left( \lambda^{1/\alpha} + \mu \right)
  - \Pi_x \left[ 
\int_0^\tau dh \, \psi \left( v_{\lambda,\mu,a,b} (\xi_h) \right)
 \right] + \lambda \Pi_x \left[ \tau \right].
\end{align*}
The proposition then follows using the links between Green's functions, harmonic analysis and Brownian motion (see for instance page 72 of \cite{LeGallBook} for a similar calculation), and recalling \eqref{eqn:Laplace sigma}.

The case $\lambda$ or $\mu$ equal to $0$ follows from a direct application of monotone convergence.
\end{proof}

Theorem~\ref{theorem:main} is now a rather easy consequence of Lemma~\ref{lemma:vlambdamu}:

\begin{proof}[Proof of Theorem~\ref{theorem:main}]
Since $\mathbb N_x$-a.s. one has $\{\mathcal R \subset (a,b)  \} = \{ L_\sigma = 0 \}$, we have by monotone convergence for all $x>0$
\[
v_{\lambda,a,b} (x ) = \lim_{\mu \to + \infty} v_{\lambda,\mu, a, b} (x ).
\]
Since set of non-negative solutions of the ODE $\frac{1}{2} v'' = v^\alpha - \lambda$ is closed under pointwise convergence (see Proposition 9 in Section V.3 of \cite{LeGallBook}), we deduce from \cref{lemma:vlambdamu} that $v_{\lambda,a,b} $ satisfies the same ODE. Calculating the boundary values of $v_{\lambda,a,b}$ is straightforward, as well as the fact that the ODE with these boundary values has a unique non-negative solution.
\end{proof}

The case $a=0$ and $b = + \infty$ is of special interest:

\begin{corollary}
For $\lambda \geq 0$ and $x>0$ we set
\begin{equation}
v_{\lambda} (x) = \mathbb N_x \left[ 1 - \indic{0 \notin \mathcal R } \exp(-\lambda \sigma)\right].
\end{equation}
The mapping $x \mapsto v_{\lambda} (x)$ is the unique non-negative solution of the ODE
\begin{equation} \label{eq:ODEvlamda}
\begin{cases}
\frac{1}{2} v'' = v^\alpha - \lambda \quad \text{on $(0,+\infty)$},\\
\lim_{x \to 0+} v(x) = + \infty,\\
\lim_{x \to +\infty} v(x) = \lambda^{1/\alpha}.
\end{cases}
\end{equation}
\end{corollary}

The special case $\lambda =0$ of ODE~\eqref{eq:ODEvlamda} gives the first statement of Theorem~\ref{thm:range snake intro}, for $x >0$ (this can be verified by direct substitution):
\begin{equation}\label{eq:rangesnake}
\mathbb N_x \left[ \indic{0 \in \mathcal R } \right] = \left( \frac{\alpha + 1}{(\alpha - 1)^2} \, \frac{1}{x^2} \right)^{\frac{1}{\alpha -1}}.
\end{equation}
Unfortunately, the solution of ODE~\eqref{eq:ODEvlamda} does not have an explicit expression for $\lambda \neq 0$. However, we can obtain its asymptotic behaviour near $0$ using only elementary calculus:
\begin{corollary}\label{cor:v(x) expansion}
For every $\lambda>0$, we have, as $x \to 0^+$:
\begin{equation}\label{eq:v0expr}
v_{\lambda} (x) = \left( \frac{\alpha + 1}{(\alpha - 1)^2} \, \frac{1}{x^2}\right)^{\frac{1}{\alpha -1}} \, \left( 1 + 
\frac
{\lambda \, (\alpha -1)^{\frac{2 \alpha}{\alpha -1}}}
{(3 \alpha -1) (\alpha + 1)^{\frac{1}{\alpha -1}}}
\, x^{\frac{2 \alpha}{\alpha -1}} + o \left( x^{\frac{2 \alpha}{\alpha -1}} \right) \right).
\end{equation}
\end{corollary}
\begin{proof}
In the whole proof, $\lambda >0$ is fixed. For $y > \lambda^{1/\alpha}$, we set
\[
F(y) = \int_y^\infty \frac{du}{2 \sqrt{ \frac{u^{\alpha + 1}}{\alpha +1} - \lambda u + \frac{\alpha}{\alpha +1} \lambda^{\frac{\alpha +1}{\alpha}}}}.
\]
One can check that this function is well-defined, non-negative, decreasing, that $\lim_{y \to +\infty} f(y) = 0$ and that $\lim_{y \downarrow \lambda^{1/\alpha}} f(y) = + \infty$. We now set for $x >0$
\[
v(x) = F^{-1}(x).
\]
A direct calculation shows that the function $v$ is then the unique non-negative solution of \eqref{eq:ODEvlamda} and hence we can obtain an asymptotic expansion for $v$ at $0^+$ through an expansion for $F$ at $+ \infty$.

Using a Taylor expansion for the integrand and then integrating we see that, as $y \to +\infty$,
\[
F(y) = \frac{(\alpha + 1)^{1/2}}{\alpha -1} \, \frac{1}{y^{\frac{\alpha -1}{2}}} + \frac{\lambda \, (\alpha + 1)^{3/2}}{2 (3 \alpha -1)} \, \frac{1}{y^{\frac{3 \alpha -1}{2}}} + o \left( \frac{1}{y^{\frac{3 \alpha -1}{2}}} \right).
\]
Using this expansion in the identity $F(v(x)) = x$ for $x >0$ yields the desired result.
\end{proof}

\section{Computation of moments}

In this section we calculate some of the moments of the range of the snake, following the calculations of Delmas \cite{delmas2003computation} in the Brownian case.
Define
\[
R=\sup \{x \in \R: x \in \mathcal{R}\}, \quad L=\inf \{x \in \R: x \in \mathcal{R}\},
\]
and for $\lambda \geq 0$,
\begin{align*}
H_{\lambda}(x) = \int_0^{\infty}e^{-\lambda s} \frac{\N^{(1)}\left[ R \geq xs^{-\frac{\alpha-1}{2\alpha}}  \right]}{\alpha \Gamma(1 - \frac{1}{\alpha}) s^{\frac{1}{\alpha}+1}} ds.
\end{align*}

We start with a lemma, analogous to Proposition~1 of \cite{delmas2003computation}.

\begin{lemma} \label{lem:almost Laplace transform range} 
For any $\lambda\geq 0$,
\begin{align*}
H_{\lambda}(x) = v_{\lambda}(x) - \lambda^{1/\alpha}.
\end{align*}
\end{lemma}
\begin{proof}
Take $\lambda>0$. Using the \Ito scaling property (point (ii) just above \eqref{eqn:Ito measure integrate s}) and  \eqref{eqn:Laplace sigma} gives
\begin{align*}
v_{\lambda} (x) = \mathbb N_0 \left[ 1 - \indic{R<x } \exp(-\lambda \sigma)\right] &= \int_0^{\infty} \frac{\N^{(s)}\left[ 1 - \indic{R<x } \exp(-\lambda \sigma)\right]}{\alpha \Gamma(1 - \frac{1}{\alpha}) s^{\frac{1}{\alpha}+1}} ds \\
&=  \int_0^{\infty} \frac{\N^{(1)}\left[ 1 - \indic{R<xs^{-\frac{2\alpha}{\alpha-1}} } \exp(-\lambda s)\right]}{\alpha \Gamma(1 - \frac{1}{\alpha}) s^{\frac{1}{\alpha}+1}} ds \\
&=  \int_0^{\infty} \frac{\left[ 1 - \exp(-\lambda s)\right]}{\alpha \Gamma(1 - \frac{1}{\alpha}) s^{\frac{1}{\alpha}+1}} ds +  \int_0^{\infty}e^{-\lambda s} \frac{\N^{(1)}\left[ R \geq xs^{-\frac{\alpha-1}{2\alpha}}  \right]}{\alpha \Gamma(1 - \frac{1}{\alpha}) s^{\frac{1}{\alpha}+1}} ds \\
&=  \lambda^{1/\alpha} +  \int_0^{\infty}e^{-\lambda s} \frac{\N^{(1)}\left[ R \geq xs^{-\frac{\alpha-1}{2\alpha}}  \right]}{\alpha \Gamma(1 - \frac{1}{\alpha}) s^{\frac{1}{\alpha}+1}} ds.
\end{align*}
The same holds for $\lambda=0$ using monotone convergence.
\end{proof}

This allows us to bound and compute some of the moments of $R$ in the following two corollaries.

\begin{corollary}
For all $\beta>0$, $\Noo (R^{\beta}) < \infty$. Moreover, for $\beta> \frac{2}{\alpha-1}$, 
\[
\Noo (R^{\beta}) = \frac{\alpha\Gamma\left( 1 -\frac{1}{\alpha}\right)\beta C_{\beta}}{\Gamma \left(\frac{\beta (\alpha-1)-2}{2\alpha}\right)},
\]
where $ C_{\beta} =\int_0^{\infty} y^{\beta-1}(v_{1}(y)-1) dy$.
\end{corollary}
\begin{proof}
We emulate the proof of \cite[Corollary 3]{delmas2003computation}. Multiplying $H_{\lambda}(x)$ by $x^{\beta-1}$ and integrating both sides of Lemma \ref{lem:almost Laplace transform range} over $x$ (here one can make the substitution $y=xs^{-\frac{\alpha-1}{2\alpha}}$ on the left hand side) gives, for any $\lambda>0$,

\begin{align}\label{eqn:int lambda moments}
\frac{\Noo (R^{\beta})}{\beta\alpha\Gamma\left(1-\frac{1}{\alpha}\right)} \int_0^{\infty} r^{\frac{\beta (\alpha-1)}{2\alpha}-\frac{\alpha+1}{\alpha}} e^{-\lambda r} dr  = \int_0^{\infty} y^{\beta-1}(v_{\lambda}(y)-\lambda^{1/\alpha}) dy,
\end{align}
so it is sufficient to show that this latter integral is finite (the claimed upper bound then follows by taking $\lambda=1$). For this we use the representation 
\[
v_{\lambda}(x) = F_{\lambda}^{-1}(x),
\]
where $F_{\lambda}=F$ in the notation of Corollary \ref{cor:v(x) expansion}. We claim that for any $\lambda>0$, there exists $C \in (0, \infty)$ such that
\begin{align}\label{eqn:F asymp}
F_{\lambda}(\lambda^{1/\alpha}+\epsilon) \sim c \log \epsilon^{-1}
\end{align}
as $\epsilon \downarrow 0$. In particular this implies that $
F_{\lambda}^{-1}(c \log \epsilon^{-1})  \sim (\lambda^{1/\alpha}+\epsilon) $, or, changing parametrisation, 
\begin{align} \label{eq:Hexp}
H_\lambda(K) = v_{\lambda}(K) - \lambda^{1/\alpha} = F_{\lambda}^{-1}(K) - \lambda^{1/\alpha} \sim e^{-K/c}
\end{align}
as $K \uparrow \infty$. Hence $v_\lambda(y) - \lambda^{1/\alpha}$ decays exponentially in $y$ as $y \to \infty$. In addition, when $y \to 0^+$, Corollary~\ref{cor:v(x) expansion} gives that $v_\lambda(y)$ is of order $y^{-\frac{2}{\alpha -1}}$, and the integral in \eqref{eqn:int lambda moments} is thus finite for $\beta >\frac{2}{\alpha -1}$. This yields the formula of the corollary and also the finiteness of the moments for all $\beta >0$. 

To establish \eqref{eqn:F asymp}, we use the formula for $F_{\lambda}$ to deduce that, as $\epsilon \downarrow 0$ (with $\lambda>0$ fixed):
\begin{align}\label{eqn:integral for moments}
\begin{split}
F_{\lambda}(\lambda^{1/\alpha}+\epsilon) &\sim \int_{\lambda^{1/\alpha}+\epsilon}^1 \frac{du}{2 \sqrt{ \frac{u^{\alpha + 1}}{\alpha +1} - \lambda u + \frac{\alpha}{\alpha +1} \lambda^{\frac{\alpha +1}{\alpha}}}} \\
&= \lambda^{-\frac{\alpha -1}{2\alpha}} \int_{\epsilon \lambda^{-1/\alpha} }^1 \frac{dv}{2 \sqrt{ \frac{(v+1)^{\alpha + 1}}{\alpha +1} - (v+1) + \frac{\alpha}{\alpha +1}}} \\
&\sim \lambda^{-\frac{\alpha -1}{2\alpha}} \int_{\epsilon \lambda^{-1/\alpha} }^1 \frac{dv}{\sqrt{2 \alpha} v} \sim \frac{1}{\sqrt{2\alpha}\lambda^{\frac{\alpha -1}{2\alpha}}}\log (\epsilon^{-1}).\qedhere
\end{split}
\end{align}
\end{proof}

We can in addition compute the following, which includes the second statement of Theorem~\ref{thm:range snake intro}.

\begin{corollary}\label{cor:moment R}
For any $\lambda>0$,
\[
\N^{(1)}\left[ R \right]  = \frac{ C(3-\alpha)\Gamma(1 - \frac{1}{\alpha})}{2 \Gamma (\frac{3\alpha-3}{2\alpha})}  \quad \text{and} \quad \N^{(1)}[R^{\frac{2}{\alpha-1}}] = \Gamma \left(1 - \frac{1}{\alpha} \right) \left( \frac{\alpha + 1}{(\alpha - 1)^2} \right)^{\frac{1}{\alpha -1}},
\]
where $C = \int_0^{\infty} (H_{0}(x)-H_{1}(x)) dx$.
\end{corollary}
\begin{proof}
Here again, we follow the proof given by Delmas in \cite[Corollary 2]{delmas2003computation}.

\textbf{First moment.}
Take any $\lambda >0$. From our expression for $v_0$ and Corollary~\ref{cor:v(x) expansion}, the function $H_{0}(x)-H_{\lambda}(x)$ is bounded for $x>0$ and the bound~\eqref{eq:Hexp} yields that it is of order $1/x^2$ when $x\to \infty$. Therefore, the integral
\[
\int_0^{\infty} (H_{0}(x)-H_{\lambda}(x)) dx
\]
is finite. By Fubini's theorem, setting $y=xs^{-\frac{\alpha-1}{2\alpha}} $ then $r=\lambda s$ gives
\begin{align*}
\int_0^{\infty} (H_{0}(x)-H_{\lambda}(x)) dx &= \int_0^{\infty} \int_0^{\infty}(1-e^{-\lambda s}) \frac{\N^{(1)}\left[ R \geq xs^{-\frac{\alpha-1}{2\alpha}}  \right]}{\alpha \Gamma(1 - \frac{1}{\alpha}) s^{\frac{1}{\alpha}+1}} ds dx \\
&= \int_0^{\infty} \int_0^{\infty} s^{\frac{\alpha-1}{2\alpha}}(1-e^{-\lambda s}) \frac{\N^{(1)}\left[ R \geq y \right]}{\alpha \Gamma(1 - \frac{1}{\alpha})  s^{\frac{1}{\alpha}+1}} ds dy \\
&= \frac{1}{\alpha \Gamma(1 - \frac{1}{\alpha})}\lambda^{\frac{3-\alpha}{2\alpha}} \N^{(1)}\left[ R \right] \int_0^{\infty} r^{\frac{-\alpha-3}{2\alpha}}(1-e^{-r})  dr \\
&= \frac{2}{(3-\alpha)\Gamma(1 - \frac{1}{\alpha})}\lambda^{\frac{3-\alpha}{2\alpha}} \N^{(1)}\left[ R \right] \Gamma \left(\frac{3\alpha-3}{2\alpha}\right).
\end{align*}

\textbf{The $\frac{2}{\alpha-1}$ moment.} By \cref{lem:almost Laplace transform range}, $H_0(x) = v_0(x)$ and so
\begin{align*}
\int_0^{\infty}\frac{\N^{(1)}\left[ R \geq s^{-\frac{\alpha-1}{2\alpha}}  \right]}{\alpha \Gamma(1 - \frac{1}{\alpha}) s^{\frac{1}{\alpha}+1}} ds = \left( \frac{\alpha + 1}{(\alpha - 1)^2} \right)^{\frac{1}{\alpha -1}}.
\end{align*} Substituting $u= s^{-\frac{\alpha-1}{2\alpha}} $ gives
\begin{align*}
\int_0^{\infty}\frac{2\alpha}{\alpha-1}u^{-\frac{3\alpha-1}{\alpha-1}}\frac{\N^{(1)}\left[ R \geq u  \right]}{\alpha \Gamma(1 - \frac{1}{\alpha}) u^{-\frac{2(1+\alpha)}{\alpha-1}}} du = 
\int_0^{\infty} \frac{2}{\alpha-1}u^{\frac{3-\alpha}{\alpha-1}}\frac{\N^{(1)}\left[ R \geq u  \right]}{ \Gamma(1 - \frac{1}{\alpha}) } du = \left( \frac{\alpha + 1}{(\alpha - 1)^2} \right)^{\frac{1}{\alpha -1}}.
\end{align*}
Integrating by parts then gives $\N^{(1)}[R^{\frac{2}{\alpha-1}}] = {\Gamma(1 - \frac{1}{\alpha})} \left( \frac{\alpha + 1}{(\alpha - 1)^2} \right)^{\frac{1}{\alpha -1}}$.
\end{proof}

We now study some two-sided estimates. Recall
\[
v_{\lambda, a,b}(0) = \mathbb N_0 \left[ 1 - \indic{\mathcal R \subset (a,b) } \exp(-\lambda \sigma)\right] = \mathbb N_0 \left[ 1 - \indic{R<b, L>a } \exp(-\lambda \sigma)\right],
\]
and set
\begin{align*}
I_{\lambda}(a,b) :&= \int_0^{\infty}e^{-\lambda s} \frac{\N^{(1)}\left[ R \geq bs^{-\frac{\alpha-1}{2\alpha}} \text{ or } L \leq as^{-\frac{\alpha-1}{2\alpha}}  \right]}{\alpha \Gamma(1 - \frac{1}{\alpha}) s^{\frac{1}{\alpha}+1}} ds, \\
J_{\lambda}(a,b) :&= \int_0^{\infty}e^{-\lambda s} \frac{\N^{(1)}\left[ R \geq bs^{-\frac{\alpha-1}{2\alpha}} \text{ and } L \leq as^{-\frac{\alpha-1}{2\alpha}}  \right]}{\alpha \Gamma(1 - \frac{1}{\alpha}) s^{\frac{1}{\alpha}+1}} ds.
\end{align*}
Note that, by symmetry, we have for any $\lambda \geq 0$ that
\begin{align}\label{eqn:HIJ relation}
I_{\lambda}(-b,b) = 2H_{\lambda}(b) - J_{\lambda}(-b,b).
\end{align}
 
Then exactly the same proof as in \cref{lem:almost Laplace transform range} gives the following.

\begin{lemma}\label{lem:sort of Laplace double}
For $a<0< b$:
\begin{align*}
I_{\lambda}(a,b) = v_{\lambda,a ,b}(0) - \lambda^{1/\alpha}.
\end{align*}
\end{lemma}

As a result, we deduce the following.

\begin{corollary}
\begin{align*}
\Noo_0 (\min(R, -L)^{\frac{2}{\alpha-1}}) = \Gamma \left(1 - \frac{1}{\alpha}\right)\left[ 2\left( \frac{\alpha + 1}{(\alpha - 1)^2} \,  \right)^{\frac{1}{\alpha -1}} - \alpha_0^{\frac{2}{\alpha-1}}\right] ,
\end{align*}
where $\alpha_0 = \frac{\sqrt{\alpha+1}}{2}\int_{1}^{\infty} \frac{du}{\sqrt{u^{\alpha+1} -1}} $.
\end{corollary}
\begin{proof}
We follow the proof of \cite[Proposition 5]{delmas2003computation} which treats the Brownian case. For $x \in (0, 2r)$ we set $w_{\lambda,r}(x) = v_{\lambda, 0, 2r}(x)$. Note that, by symmetry, we have for all $a,b>0$ that 
\[
v_{\lambda,a,b}(0) = w_{\lambda,\frac{b-a}{2}}(\min(-a,b)) \quad \text{and} \quad  w_{\lambda,r}'(r)=0.
\]
Since $w_{0,r}$ is a solution of ODE~\eqref{eqn:ODE ab local mu infinite}, one can then verify by direct substitution that $w_{0,r}(t) = G^{-1}(t)$ where
\begin{align}\label{eqn:G def}
G(y) = \int_y^{\infty} \frac{du}{2\sqrt{\frac{u^{\alpha+1}}{\alpha+1} - \frac{w_{0,r}(r)^{\alpha+1}}{\alpha+1}}}.
\end{align}
To prove the corollary, we will use \eqref{eqn:G def} to evaluate $w_{0,r}(r)$, Lemma \ref{lem:sort of Laplace double} to evaluate $I_0(-r,r)$ and then \eqref{eqn:HIJ relation} to evaluate $J_0(-r,r)$. We will then proceed similarly to the proof of \cref{cor:moment R} to evaluate the $\frac{2}{\alpha-1}^{th}$ moment. To this end, fix $r>0$, let $\theta=w_{0,r}(r)$ and note that by \eqref{eqn:G def} we have for any $t>0$ that
\begin{align*}
\int_{w_{0,r}(t)}^{\infty} \frac{du}{2\sqrt{\frac{u^{\alpha+1}}{\alpha+1} - \frac{\theta^{\alpha+1}}{\alpha+1}}} = t.
\end{align*}
Substituting $v = \frac{u}{\theta}$ and then taking $t=r$ we deduce that 
\begin{align*}
\frac{\sqrt{\alpha+1}}{2}\int_{1}^{\infty} \frac{dv}{\sqrt{v^{\alpha+1} -1}} = r\theta^{\frac{\alpha-1}{2}},
\end{align*}
and hence it follows from Lemma \ref{lem:sort of Laplace double} (recall the definition of $\alpha_0$ in the statement of the corollary) that
\begin{align*}
I_0(-r,r) = w_{0,r}(r) = \theta = \left( \frac{\alpha_0}{r}\right)^{\frac{2}{\alpha-1}}.
\end{align*}
Hence, by \eqref{eqn:HIJ relation}, \cref{lem:almost Laplace transform range} and \eqref{eq:rangesnake} we have
\begin{align*}
J_0(-r,r) = 2H_0(r) - I_0(-r,r) 
=\left[ 2\left( \frac{\alpha + 1}{(\alpha - 1)^2} \,  \right)^{\frac{1}{\alpha -1}} - \alpha_0^{\frac{2}{\alpha-1}}\right] r^{-\frac{2}{\alpha-1}}.
\end{align*}
Now proceeding as in the proof of \cref{cor:moment R} we see that
\begin{align*}
J_{0}(-1,1) &= \int_0^{\infty} \frac{\N^{(1)}\left[ \min (R, -L) \geq s^{-\frac{\alpha-1}{2\alpha}} \right]}{\alpha \Gamma(1 - \frac{1}{\alpha}) s^{\frac{1}{\alpha}+1}} ds = \frac{2\alpha}{\alpha-1} \int_0^{\infty} t^{\frac{2}{\alpha-1}-1} \frac{\N^{(1)}\left[ \min (R, -L) \geq t \right]}{\alpha \Gamma(1 - \frac{1}{\alpha}) } dt \\
&= \frac{1}{\Gamma(1 - \frac{1}{\alpha})}  \N^{(1)}\left[ (\min (R, -L))^{\frac{2}{\alpha-1}} \right].
\end{align*}
Tying everything together gives the result.
\end{proof}

\section{Occupation measure}\label{sctn:occ measure}

Let $\Ta$ be an $\alpha$-stable tree with total mass $1$, that is a tree sampled according to $N^{(1)}$. Recall that the uniform measure $\lambda$ on $\Ta$ is the pushforward of Lebesgue measure on $[0,1]$ by the canonical projection $p : [0,1] \to \Ta$.
We recall from \eqref{eqn:ISE def} and \eqref{eqn:ISE def rerooted} that we defined the random measure $\mathcal I$ of the stable snake on $\mathbb R$ and its version shifted by the infimum of the snake by 
\begin{equation*}
\langle \mathcal I , f \rangle = \int_{\Ta} \lambda(da) \, f(Z_a) \qquad \text{and} \qquad \langle \overline{\mathcal I} , f \rangle = \int_{\Ta} \lambda(da) \, f(\overline Z_a).
\end{equation*}

The aim of this section is to calculate the first moment
\[
\mathbb N_0 \left( (1 - e^{-\lambda \sigma}) \, \overline{\mathcal I} \left( [0,\varepsilon] \right) \right),
\]
where $\overline{\mathcal I} \left( [0,\varepsilon] \right)= \langle \overline{\mathcal I} , \indic{[0,\varepsilon]} \rangle$, and to give a tight upper bound as $\varepsilon \downarrow 0$. For the Brownian case (in which case $\alpha =2$), this was done in \cite[Section 3]{legall-weill} and our techniques will be similar, albeit a bit more technical.

\bigskip

Since our calculations will involve Bessel processes, let us introduce some notation. For every $d >0$ and $x >0$, we will denote a $d$-dimensional Bessel process started at $x$ by $(R_t)_{t \geq 0}$. We also denote the law of this process by $P_x^{(d)}$. Absolute continuity relations between Bessel processes have been established by Yor~\cite{Yor}. The special case we will use in this work can be found in \cite[Exercise XI.1.22]{Revuz-Yor} and \cite[Proposition 2.6]{legall-weill}:

\begin{proposition} \label{prop:Bessel}
Let $t >0$ and $F$ be a nonnegative measurable function on $C([0,t],\mathbb R)$. Then, for every $x>0$ and $\tilde{\lambda} >0$, writing $\tilde{\nu} = \sqrt{\tilde{\lambda}^2 + \frac{1}{4}}$, one has
\[
\Pi_x \left[\indic{\xi[0,t] \subset (0,\infty) } \exp \left( - \frac{\tilde{\lambda}^2}{2} \int_0^t \frac{dr}{\xi_r^2}\right) \, F \left( (\xi_r)_{0\leq r \leq t}\right) \right]
= 
x^{\tilde{\nu} +\frac{1}{2}} \, E_x^{(2(\tilde{\nu}+1))} \left[ (R_t)^{-\tilde{\nu} - \frac{1}{2}} F \left( (R_r)_{0\leq r \leq t}\right)\right].
\]
\end{proposition}

We will also use a classic stochastic monotonicity result regarding Bessel processes: if $x<y$, then for any $t \geq 0$ the law of $R_t$ under $P_x^{(d)}$ is stochastically dominated by the law of $R_t$ under $P_y^{(d)}$.

\bigskip

We are now ready to state our formula for the first moment of the occupation measure $\overline{\mathcal I}$ in terms of Bessel processes:

\begin{proposition} \label{prop:momentI}
Fix $\lambda >0$ and recall that the function $v_\lambda$ is the unique solution to the ODE~\eqref{eq:ODEvlamda}.
For every $\varepsilon >0$, one has that $\mathbb N_0 \left( (1 - e^{-\lambda \sigma}) \, \overline{\mathcal I} \left( [0,\varepsilon] \right) \right)$ is equal to
\begin{equation}
\varepsilon^{\frac{2 \alpha}{\alpha -1}} \int_0^{\infty} da \,  E_\varepsilon^{(\frac{5 \alpha - 1}{\alpha -1})} \left[
R_a^{-\frac{2 \alpha}{\alpha -1}} \, \left( 1 - \exp \left\{ - \int_0^a dt \, \left( \psi' \left( v_\lambda(R_t) \right) - \frac{\alpha (\alpha +1)}{(\alpha-1)^2} \frac{1}{R_t^2}\right) \right\} \right)
\right].
\end{equation}
\end{proposition}
\begin{proof}
The proof is similar to that of \cite[Lemma 3.2]{legall-weill}. Denote by $\underline{W}$ the minimum of the snake $W$ and by $W^{[s]}$ the snake $W$ rerooted at $s$. By definition one has
\begin{align*}
\mathbb N_0 \left( (1 - e^{-\lambda \sigma}) \, \overline{\mathcal I} \left( [0,\varepsilon] \right) \right)
&=
\mathbb N_0 \left( (1 - e^{-\lambda \sigma}) \, \int_0^\sigma ds \, \indic{ \widehat W_s - \underline{W} \leq \varepsilon} \right) \\ 
& = \mathbb N_0 \left( (1 - e^{-\lambda \sigma}) \, \int_0^\sigma ds \, \indic{ \underline{W}^{[s]} \geq - \varepsilon} \right)  = \mathbb N_0 \left( (1 - e^{-\lambda \sigma}) \, \int_0^\sigma ds \, \indic{ \underline{W} \geq - \varepsilon} \right),
\end{align*}
where we used the rerooting invariance of Proposition~\ref{prop:uniform rerooting} for the last equality.

We will use Proposition~\ref{prop:firstmoment ds} to calculate this quantity. Recall that, for $h>0$, under $E^0\otimes\Pi_0$, conditionally on the measure $J_h+\hat J_h$ and the Brownian trajectory $(\xi_a)_{0 \leq a \leq h}$, we are given a Poisson point process $\sum_{i\in I} \delta_{a_i,\omega_i}$ with intensity $(J_h + \hat J_h)(da) \otimes \mathbb N_{\xi_a} (dw)$. We denote this by $\Pcal$ and its law and expectation by $\Pb_{\Pcal}$ and $\mathbb{E}_{\Pcal}$. Proposition~\ref{prop:firstmoment ds} and \cref{prop:Poisson master formula} then give
\begin{align*}
\mathbb N_0 &\left( (1 - e^{-\lambda \sigma}) \, \int_0^\sigma ds \, \indic{ \underline{W} \geq - \varepsilon} \right)\\
&=
\int_0^\infty dh \, E^0\otimes \Pi_0 \left[ \mathbb{E}_{\Pcal}\left[
\indic{\underline{\xi}_h>-\varepsilon} \,
\indic{\forall i , -\varepsilon -\xi_{a_i} \notin \mathcal R(w_i)}
\left( 1 - e^{- \lambda \sum_{i\in I} \sigma(w_i)}\right)
\middle| J_h+\hat J_h, (\xi_a)_{0 \leq a \leq h} \right]\right],\\
&=
\int_0^\infty dh \, E^0\otimes \Pi_0 \left[
\indic{\underline{\xi}_h>-\varepsilon} \,
\exp \left( - \int_{[0,h]\times \mathcal W} \left( 1 - \indic{ -\varepsilon -\xi_{a} \notin \mathcal R(w)}\right) 
(J_h + \hat J_h(da)) \otimes \mathbb N_{\xi_a}(dw)
\right) \right.\\
& \qquad \qquad 
- \left.
\indic{\underline{\xi}_h>-\varepsilon} \,
\exp \left( - \int_{[0,h]\times \mathcal W} \left( 1 - \indic{ -\varepsilon -\xi_{a} \notin \mathcal R(w)} e^{-\lambda \sigma(w)} \right) 
(J_h + \hat J_h(da)) \otimes \mathbb N_{\xi_a}(dw)
\right)
\right],\\
&=
\int_0^\infty dh \, E^0\otimes \Pi_\varepsilon \left[ 
\indic{\underline{\xi}_h>0} \,
\exp \left( - \int_{[0,h]} \mathbb N_{\xi_a} \left( 1 - \indic{ 0 \notin \mathcal R}\right) 
(J_h + \hat J_h(da))
\right) \right.\\
& \qquad \qquad \qquad \qquad \qquad
- \left.
\indic{\underline{\xi}_h>0} \,
\exp \left( - \int_{[0,h]} \mathbb N_{\xi_a} \left( 1 - \indic{ 0 \notin \mathcal R} e^{-\lambda \sigma} \right) 
(J_h + \hat J_h(da))
\right)
\right],\\
&=
\int_0^\infty dh \,\Pi_\varepsilon \left[ 
\indic{\underline{\xi}_h>0} \,
\exp \left( - \int_{[0,h]} da \, \psi' \left( v_0(\xi_a) \right)
\right) 
-
\indic{\underline{\xi}_h>0} \,
\exp \left( - \int_{[0,h]} da \, \psi' \left( v_{\lambda}(\xi_a) \right)
\right)
\right].
\end{align*}
Using the expression~\eqref{eq:rangesnake} for $v_0$ and replacing $\psi' (x)$ by $\alpha x^{\alpha-1}$ we get:
\begin{align*}
\mathbb N_0 &\left( (1 - e^{-\lambda \sigma}) \, \int_0^\sigma ds \, \indic{\underline{W} \geq - \varepsilon} \right)\\
&=
\int_0^\infty dh \,\Pi_\varepsilon \left[ 
\indic{\underline{\xi}_h>0} \,
\exp \left( - \frac{\alpha(\alpha +1)}{(\alpha -1)^2}\int_{[0,h]} \frac{da}{\xi_a^2} \right)\right. \\
& \qquad \qquad \qquad \qquad  \times
\left. \left(1
-
\exp \left( - \int_{[0,h]} da \left( \psi' \left( v_\lambda(\xi_a) \right)
- \frac{\alpha(\alpha +1)}{(\alpha -1)^2} \frac{1}{\xi_a^2}
\right)
\right)
\right)
\right].
\end{align*}
The proposition then follows from the absolute continuity relation between Bessel processes of Proposition~\ref{prop:Bessel} with $\tilde{\lambda}^2 = \frac{2 \alpha (\alpha +1)}{(\alpha -1)^2}$ and $\tilde{\nu} = \frac{3 \alpha +1}{2(\alpha -1)}$.
\end{proof}

We are now able to prove \cref{cor:I exp vol UB}.

\begin{proof}[Proof of \cref{cor:I exp vol UB}]
We start with the formula obtained in Proposition~\ref{prop:momentI}:
\begin{align*}
\varepsilon^{-\frac{2 \alpha}{\alpha -1}}  \mathbb N_0 & \left( (1 - e^{-\lambda \sigma}) \, \overline{\mathcal I} \left( [0,\varepsilon] \right) \right)\\
&=
\int_0^{\infty} da \,  E_\varepsilon^{(\frac{5 \alpha -1}{\alpha -1})} \left[
R_a^{-\frac{2 \alpha}{\alpha -1}} \, \left( 1 - \exp \left\{- \int_0^a dt \, \left( \psi' \left( v_\lambda(R_t) \right) - \frac{\alpha (\alpha +1)}{(\alpha-1)^2} \frac{1}{R_t^2}\right) \right\} \right)
\right].
\end{align*}
To bound the integral in the previous display, we start by splitting it two. Fix $\delta >0$ and set
\begin{align*}
I_1 &:= \int_0^{\delta} da \,  E_\varepsilon^{(\frac{5 \alpha-1}{\alpha -1})} \left[
R_a^{-\frac{2 \alpha}{\alpha -1}} \, \left( 1 - \exp \left\{ - \int_0^a dt \, \left( \psi' \left( v_\lambda(R_t) \right) - \frac{\alpha (\alpha +1)}{(\alpha-1)^2} \frac{1}{R_t^2}\right) \right\} \right)
\right], \\
I_2 &:= \int_\delta^{\infty} da \,  E_\varepsilon^{(\frac{5 \alpha-1}{\alpha -1})} \left[
R_a^{-\frac{2 \alpha}{\alpha -1}} \, \left( 1 - \exp \left\{ - \int_0^a dt \, \left( \psi' \left( v_\lambda(R_t) \right) - \frac{\alpha (\alpha +1)}{(\alpha-1)^2} \frac{1}{R_t^2}\right) \right\} \right)
\right].
\end{align*}
Let us first focus on $I_2$:
\begin{align*}
I_2 \leq \int_\delta^{\infty} da \,  E_\varepsilon^{(\frac{5 \alpha-1}{\alpha -1})} \left[
R_a^{-\frac{2 \alpha}{\alpha -1}} \right]
= \int_\delta^{\infty} \frac{da}{a^{\frac{\alpha}{\alpha -1}}} \,  E_{\varepsilon/\sqrt{a}}^{(\frac{5 \alpha-1}{\alpha -1})} \left[
R_1^{-\frac{2 \alpha}{\alpha -1}} \right]
\leq \int_\delta^{\infty} \frac{da}{a^{\frac{\alpha}{\alpha -1}}} \,  E_{0}^{(\frac{5 \alpha-1}{\alpha -1})} \left[
R_1^{-\frac{2 \alpha}{\alpha -1}} \right],
\end{align*}
using the scaling property and the stochastic monotonicity with respect to the starting point for Bessel processes (see the discussion directly underneath \cref{prop:Bessel}). The expectation in the integral is finite since, for every $d > 0$ and every $a > -d$ one has $E_{0}^{(d)} \left[R_1^{a} \right] < \infty$ by \cite[Equation (15)]{lawler2018notes}. Therefore, there exists a constant $c_\alpha >0$ depending only on $\alpha$ such that for every $\delta >0$,
\[
I_2 \leq c_\alpha \, \delta^{-\frac{1}{\alpha -1}}.
\]

We now turn to $I_1$. From Corollary~\ref{cor:v(x) expansion}, we know that there exists another constant $C_{\alpha,\lambda} >0$, depending only on $\alpha$ and $\lambda$, such that, for every $x >0$,
\[
0 \leq \psi' \left( v_\lambda(x) \right) - \frac{\alpha (\alpha +1)}{(\alpha-1)^2} \frac{1}{x^2} \leq  C_{\alpha,\lambda} \, x^{\frac{2}{\alpha -1}}.
\]
This gives
\[
I_1 \leq C_{\alpha,\lambda} \,\int_0^{\delta} da \, \int_0^a dt \, E_\varepsilon^{(\frac{5 \alpha-1}{\alpha -1})} \left[
R_a^{-\frac{2 \alpha}{\alpha -1}} \,  R_t^{\frac{2}{\alpha -1}}
\right]
=
C_{\alpha,\lambda} \,\int_0^{\delta} da \, \int_0^a dt \, E_\varepsilon^{(\frac{5 \alpha-1}{\alpha -1})} \left[
R_t^{\frac{2}{\alpha -1}} E_{R_t}^{(\frac{5 \alpha-1}{\alpha -1})} \left[
R_{a-t}^{-\frac{2 \alpha}{\alpha -1}} \right]
\right]
,
\]
where we used the Markov property of Bessel processes for the last identity. To conclude, we adapt arguments found at the end of the proof of Lemma 3.2 in \cite{legall-weill}. First, using scaling and stochastic monotonicity of Bessel processes we have that there exists $C<\infty$ (as above, $C$ is finite since for every $d > 0$ and every $a > -d$ one has $E_{0}^{(d)} \left[R_1^{a} \right] < \infty$) such that for any $0 < \varepsilon , \delta < 1$ and $0 < t < a < \delta$
\begin{align*}
E_\varepsilon^{(\frac{5 \alpha-1}{\alpha -1})} \left[
R_t^{\frac{2}{\alpha -1}} E_{R_t}^{(\frac{5 \alpha-1}{\alpha -1})} \left[
R_{a-t}^{-\frac{2 \alpha}{\alpha -1}} \right] \right]
& \leq 
E_\varepsilon^{(\frac{5 \alpha-1}{\alpha -1})} \left[
R_t^{\frac{2}{\alpha -1}} E_0^{(\frac{5 \alpha-1}{\alpha -1})} \left[
R_{a-t}^{-\frac{2 \alpha}{\alpha -1}} \right] \right] \\
& = E_\varepsilon^{(\frac{5 \alpha-1}{\alpha -1})} \left[
R_t^{\frac{2}{\alpha -1}}  \right] \frac{C}{(a-t)^{\frac{\alpha}{\alpha -1}}} \\
& \leq E_1^{(\frac{5 \alpha-1}{\alpha -1})} \left[
R_1^{\frac{2}{\alpha -1}}  \right] \frac{C}{(a-t)^{\frac{\alpha}{\alpha -1}}} \\
& = \frac{C'}{(a-t)^{\frac{\alpha}{\alpha -1}}},
\end{align*}
with constants $C,C' >0$ depending only on $\alpha$. Unfortunately, this bound is only good for $t$ bounded away from $a$ and we need a better bound for $t$ close to $a$. Applying Proposition~\ref{prop:Bessel} with $F=1$ gives, for every $x,r >0$ 
\[
E_{x}^{(\frac{5 \alpha-1}{\alpha -1})} \left[
R_{r}^{-\frac{2 \alpha}{\alpha -1}} \right] \leq x^{-\frac{2 \alpha}{\alpha -1}}.
\]
We get, for every $\varepsilon >0$ and $0 < t < a$
\begin{align*}
E_\varepsilon^{(\frac{5 \alpha-1}{\alpha -1})} \left[
R_t^{\frac{2}{\alpha -1}} E_{R_t}^{(\frac{5 \alpha-1}{\alpha -1})} \left[
R_{a-t}^{-\frac{2 \alpha}{\alpha -1}} \right] \right]
 \leq 
E_\varepsilon^{(\frac{5 \alpha-1}{\alpha -1})} \left[
R_t^{\frac{2}{\alpha -1}} \,
R_{t}^{-\frac{2 \alpha}{\alpha -1}}\right] 
= E_\varepsilon^{(\frac{5 \alpha-1}{\alpha -1})} \left[
R_t^{-2}  \right] = \frac{C''}{t},
\end{align*}
where the constant $C''$ depends only on $\alpha$. This finally gives the following upper bound for $I_1$, uniform over $0<\varepsilon<1$:
\begin{align*}
I_1 & \leq C_{\alpha,\lambda} \, \int_0^{\delta} da \, \int_0^a dt \,
\min \left( \frac{C'}{(a-t)^{\frac{\alpha}{\alpha -1}}} , \frac{C''}{t} \right) \\
& \leq C_{\alpha,\lambda} \, \int_0^{\delta} da \,
\left( \int_0^{a^{\frac{\alpha}{\alpha -1}}} dt \, \frac{C'}{(a-t)^{\frac{\alpha}{\alpha -1}}} + \int_{a^{\frac{\alpha}{\alpha -1}}}^a dt \, \frac{C''}{t} \right) \\
& \leq C_{\alpha,\lambda} \, \int_0^{\delta} da \,
\left( C'_1 + C_1'' \ln a \right).
\end{align*}
This finishes the proof since the integral in the last display is finite.
\end{proof}

\section{Common increase points}\label{sctn:inc points}

This section is dedicated to the proof of Theorem~\ref{thm:no common increase points intro}. We follow the same general proof strategy as Le Gall in \cite[Lemma 2.2]{LeGall2007BrownianMapTopological} although there are some technical complications due to the fact that the height process is no longer Markovian in our case. We start with a basic lemma. 

\begin{lemma}\label{lem:height holder}
\begin{enumerate}[(a)]
\item $\No$ almost surely, for any $\gamma< 1-\frac{1}{\alpha}$, the height function is H\"older continuous with parameter $\gamma$.
\item (cf \cite[Lemma 5.1]{LeGall2007BrownianMapTopological}). Take any $b \in (0, 1/2)$. Then $\Noo$ a.e. there exists $\epsilon_0> 0$ such that for every $s,t \geq 0$ with $d(s,t) \leq \epsilon_0$,
\[
|\Wt_t-\Wt_s| \leq d(s,t)^b.
\]
\end{enumerate}
\end{lemma}
\begin{proof}
\begin{enumerate}[(a)]
\item This is proved in \cite[Theorem 1.4.4]{LeGDuqMono} under $N$. This extends to $\No$ using standard arguments involving the Vervaat transform and absolute continuity relations between the bridges and excursions. Note that the roles of $\gamma$ and $\alpha$ are reversed in the statement of \cite[Theorem 1.4.4]{LeGDuqMono}.
\item This is clearly true along a single branch and extends to the whole tree using standard arguments: in particular one can apply \cite[Theorem 2.2.4]{vaart2023empirical} using the $L^p$ norm as the $\psi$-Orlicz norm there and comparing to the pseudodistance $\sqrt{\dt (s,t)}$. (See \cite[Proposition 5.2]{DLG05} for the results on packing and covering numbers of stable trees required for that argument.)
\end{enumerate}
\end{proof}

We now prove the key lemma, analogous to Lemma 5.2 in~\cite{LeGall2007BrownianMapTopological}. Recall that $\zeta_s$ denotes the lifetime of the snake at time $s$, which is the same as the height of the corresponding tree vertex. Fix $\epsilon, \delta, t >0$ and define (see Figure~\ref{fig:L,T and T'})
\begin{align*}
T_t = \inf \{s \geq t : \zeta_s= \zeta_t + \delta \}, \ \ T_t' = \inf \{s \geq t : \zeta_s= (\zeta_t - \epsilon ) \vee 0 \}.
\end{align*}
On the event $\{T_t< T_t'\}$ set
\[
L_t = \sup \{ s < T_t: \zeta_s = \zeta_t\}.
\]

\begin{figure}[h]
\centering
\includegraphics[width=8cm]{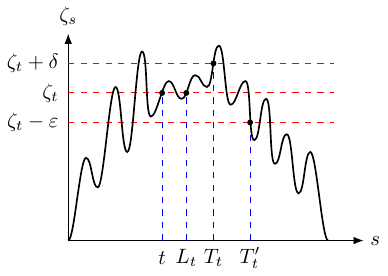}
\caption{The times $T_t$, $T'_t$ and $L_t$.}\label{fig:L,T and T'}
\end{figure}

\begin{lemma}
Under $\N_x$ choose $U$ uniformly on $[0, \sigma]$, take $\delta \in (0,1)$ and set $T=T_U, T'=T_U'$ and $L=L_U$. Then there exists $C_{\delta} < \infty$ depending only on $\alpha$ and $\delta$ such that for all $\eta \in (0,1)$ and all $x \in \R$:
\[
\N_x\left(T<T' \text{ and } \Wt_s > \Wt_L - \eta \text{ for all } s \in [L,T]\right) \leq C_{\delta} \, \epsilon \, \eta^{\frac{\alpha+1}{\alpha-1}}.
\]
\end{lemma}
\begin{proof}
The proof proceeds in two steps, one for the event $\{T<T'\}$ and the other for the event $\{\Wt_s > \Wt_L - \eta \text{ for all } s \in [L,T]\}$.

\textbf{Step 1.} We start by bounding (recall that $x$ here denotes the initial condition for the snake)
$\N_x\left(T<T'\right)$.
We henceforth assume that $\epsilon \leq \delta$ (the result for Step 1 then extends to all $\epsilon>0$ by adjusting the value of the constant $C_{\delta}$). For each $t \in [0, \sigma]$ let $P_t$ denote the direct path in $\Ta$ from the root to $p(t)$. The event $\{  T_t<T_t'\}$ is contained in the event in which at least one of the subtrees grafted to the right of $P_t$ within distance $\epsilon$ of $p(t)$ has height exceeding $\delta$. We apply \cref{prop:firstmomentrange ds} and \cref{lemma:onesidePoisson}, let $\I$ denote the indices of Poisson points for the Poisson point process $\mathcal P$ appearing in \cref{lemma:onesidePoisson}, and let $\I_{\epsilon} \subset \I$ denote those indices with height $h \in [\zeta_t, \zeta_t - \epsilon]$. We let $\Pb_{\Pcal}$ and $\mathbb{E}_{\Pcal}$ denote the law and expectation with respect to $\mathcal P$. We can then write
\begin{align*}
\N_x\left( T<T'\right) &= \N_x\left( \frac{1}{\sigma}\int_0^{\sigma} \indic{T_t<T_t'} dt\right) \\
 &\leq \int_0^{\infty} E^0\left[ \estart{
 \frac{1}{\sum_{i \in \mathcal{I} } \sigma_i} 
 \indic{ \max_{i \in \mathcal{I_{\epsilon}}} H (\T_i) > \delta} \middle | J_h}{\mathcal P} \right] dh.
 \end{align*}

We will establish separate upper bounds for the previous integral for $h < 1$ and $h > 1$. Let us start with the latter. Since $\mathcal I$ is a Poisson point process, the two parts $\I_{\epsilon}$ and $\I \setminus \I_{\epsilon}$ are independent and we can factorise the integral:
\begin{align*}
\int_1^{\infty} & E^0\left[ \estart{
 \frac{1}{\sum_{i \in \mathcal{I} } \sigma_i} 
 \indic{ \max_{i \in \mathcal{I_{\epsilon}}} H (\T_i) > \delta} \middle | J_h}{\mathcal P} \right] dh\\
 &\qquad \leq 
 \int_1^{\infty} E^0\left[ \estart{
 \frac{1}{\max_{i \in \mathcal{I} \setminus \mathcal{I_{\epsilon}}}  \sigma_i} 
 \indic{ \max_{i \in \mathcal{I_{\epsilon}}} H (\T_i) > \delta} \middle | J_h}{\mathcal P} \right] dh\\
 &\qquad =
 \int_1^{\infty} E^0\left[ \estart{
 \frac{1}{\max_{i \in \mathcal{I} \setminus \mathcal{I_{\epsilon}}}  \sigma_i} \middle | J_h}{\mathcal P}
\mathbb P_{\mathcal P} \left[\max_{i \in \mathcal{I_{\epsilon}}} H (\T_i) > \delta \middle | J_h \right] \right] dh\\
& \qquad =
 \int_1^{\infty} E^0\left[ \estart{
 \frac{1}{\max_{i \in \mathcal{I} \setminus \mathcal{I_{\epsilon}}}  \sigma_i} \middle | J_h}{\mathcal P} \right]
 E^0\left[ 
\mathbb P_{\mathcal P} \left[\max_{i \in \mathcal{I_{\epsilon}}} H (\T_i) > \delta \middle | J_h \right] \right] dh,
\end{align*}
where the last equality comes from the fact that, under $E^0$, the two conditional quantities in the previous display depend only on $U^{(1)} (t)$ respectively for $t < \varepsilon$ and $\varepsilon <t<h$.

We can compute separately the two conditional quantities in the integral. For $x>0$ and $h>1 >\varepsilon$ one has
\begin{align*}
\mathbb P_{\mathcal P} \left[\max_{i \in \mathcal I \setminus \mathcal{I_{\epsilon}}} \sigma_i \leq x \middle | J_h \right]
&=
\exp - \int_\varepsilon ^h N \left( \sigma > x \right) \, U^{(1)}(dt) = \exp \{- c_\alpha x^{-1/\alpha} U^{(1)} (h- \varepsilon)\}.
\end{align*}
This then gives
\begin{align*}
E^0\left[ \estart{
 \frac{1}{\max_{i \in \mathcal{I} \setminus \mathcal{I_{\epsilon}}}  \sigma_i} \middle | J_h}{\mathcal P} \right]
& = E^0\left[
\int_0^\infty \frac{dx}{x} \, c_\alpha \, x^{-\frac{1}{\alpha} - 1} U^{(1)}(h - \varepsilon) \, \exp \{- c_\alpha x^{-1/\alpha} U^{(1)} (h- \varepsilon)\}
\right]\\
& = E^0\left[
\left(  c_\alpha \,  U^{(1)}(h - \varepsilon) \right)^{-\alpha} \int_0^\infty dy \, y^{-2-\frac{1}{\alpha}} e^{- y^{-1/\alpha}}
\right]\\
& = c'_\alpha h^{-\frac{\alpha}{\alpha -1}}.
\end{align*}

For the second quantity we have
\begin{align*}
 E^0\left[ 
\mathbb P_{\mathcal P} \left[\max_{i \in \mathcal{I_{\epsilon}}} H (\T_i) > \delta \middle | J_h \right] \right]
& =
E^0\left[ 1-\exp \left\{ -\int_0^{\epsilon} N (H>\delta) U^{(1)} (dt) \right\} \right]\\
& = E^0\left[ 1-\exp - c_\delta U^{(1)} (\varepsilon)\right]\\
& = 1 - e^{- \varepsilon \alpha c_\delta^{\alpha -1}}\\
& \leq c'_\delta \, \varepsilon,
\end{align*}
where $c_\delta$ depends only on $\delta$ and $\alpha$. Putting these last two bounds together we get
\begin{align}
\int_1^{\infty}  E^0\left[ \estart{
 \frac{1}{\sum_{i \in \mathcal{I} } \sigma_i} 
 \indic{ \max_{i \in \mathcal{I_{\epsilon}}} H (\T_i) > \delta} \middle | J_h}{\mathcal P} \right] dh
 \leq c_1 \, \varepsilon \int_{1}^\infty dh \, h^{-\frac{\alpha}{\alpha -1}} = c \varepsilon, \label{eq:bound1}
\end{align}
with $c>0$ depending only on $\alpha$ and $\delta$.

\bigskip

We now turn to the integral over $h$ between $0$ and $1$. Heuristically, in the integral, we are considering the event where one of the subtrees of $\mathcal I_{\varepsilon}$ has height larger than $\delta$. This forces the lifetime of the corresponding excursion to be large, and the integral will be small. More precisely we have:
\begin{align*}
\int_0^{1} & E^0\left[ \estart{
 \frac{1}{\sum_{i \in \mathcal{I} } \sigma_i} 
 \indic{\max_{i \in \mathcal{I_{\epsilon}}} H (\T_i) > \delta} \middle | J_h}{\mathcal P} \right] dh\\
 &\qquad \leq 
 \int_0^{1} E^0\left[ \estart{
 \frac{1}{\max_{i \in \mathcal{I_{\epsilon}}}  \sigma_i} 
 \indic{ \max_{i \in \mathcal{I_{\epsilon}}} H (\T_i) > \delta} \middle | J_h}{\mathcal P} \right] dh\\
&\qquad = \int_0^{1} E^0\left[ 
\int_0^\infty dx \,
\mathbb P_{\mathcal P} \left(
 \frac{1}{\max_{i \in \mathcal{I}_{\epsilon} } \sigma_i} 
 \indic{ \max_{i \in \mathcal{I_{\epsilon}}} H (\T_i) > \delta}
 > x
 \middle | J_h \right) \right] dh\\
 &\qquad =
 \int_0^{1} E^0\left[ 
\int_0^\infty dx \, \left( 
\mathbb P_{\mathcal P} \left(
\max_{i \in \mathcal{I}_{\epsilon}}  \sigma_i < \frac{1}{x}
 \middle | J_h \right) 
-
\mathbb P_{\mathcal P} \left(
\max_{i \in \mathcal{I}_{\epsilon} } \sigma_i < \frac{1}{x}
\, , \, \max_{i \in \mathcal{I}_{\epsilon} } H(\T_i) < \delta
 \middle | J_h \right) 
\right)
 \right] dh\\
&\qquad =
\int_0^1 \int_0^\infty E^0 \left[ e^{-N(\sigma > 1/x) U^{(1)}(h \wedge \varepsilon)} -e^{-N(\sigma > 1/x \, \text{ or } \, H > \delta) U^{(1)}(h \wedge \varepsilon)} \right] \, dx \, dh\\
&\qquad =
\int_0^1 \int_0^\infty \left( e^{-N(\sigma > 1/x)^{\alpha -1} \cdot (h \wedge \varepsilon)} -e^{-N(\sigma > 1/x \, \text{ or }  \, H > \delta)^{\alpha -1} \cdot (h \wedge \varepsilon)} \right) \, dx \, dh.
\end{align*}

As $\varepsilon \to 0$, we have 
\begin{align*}
e^{-N(\sigma > 1/x)^{\alpha -1} \cdot (h \wedge \varepsilon)} -e^{-N(\sigma > 1/x \, \text{ or } \, H > \delta)^{\alpha -1} \cdot (h \wedge \varepsilon)}
& \sim (h \wedge \varepsilon) \cdot \left( N(\sigma > 1/x \, \text{ or } \, H > \delta)^{\alpha -1} - N(\sigma > 1/x)^{\alpha -1} \right) \\
& \leq \varepsilon  N(\sigma < 1/x \, , \, H > \delta)^{\alpha -1}.
\end{align*}
As a consequence, for $\varepsilon$ small enough, we have
\begin{align}
\int_0^{1} & E^0\left[ \estart{
 \frac{1}{\sum_{i \in \mathcal{I} } \sigma_i} 
 \indic{ \max_{i \in \mathcal{I_{\epsilon}}} H (\T_i) > \delta } \middle | J_h}{\mathcal P} \right] dh
\leq 
2 \, \varepsilon \int_0^\infty  N(\sigma < 1/x \, , \, H > \delta)^{\alpha -1} \, dx. \label{eq:bound2}
\end{align}
The integral in the last display is finite since $ N(\sigma < 1/x \, , \, H > \delta)$ has a tail of the form (see for example \cite[Lemma 4.1]{duquesne2005hausdorff})
\begin{align*}
N(\sigma < 1/x \, , \, H > \delta) \leq C_{\alpha,\delta} \, e^{ - C_{\alpha,\delta} x^{1 - \frac{1}{\alpha}}}
\end{align*}
as $x \to \infty$.

Putting our bounds \eqref{eq:bound1} and \eqref{eq:bound2} together, we finally have
\[
\N_x\left(T<T'\right) \leq c \, \varepsilon,
\]
where the constant $c$ depends only on $\delta$ and $\alpha$.

\textbf{Step 2.} The next step is to bound the probability of the event $\{ \Wt_s > \Wt_L - \eta \text{ for all } s \in [L,T] \}$, conditionally on the event $\{ T<T' \}$, under the law $\Pb^*$. For this we apply \cite[Theorem 4.6.2]{LeGDuqMono} and \cite[Theorem 3.1]{abraham2009williams}.

Recall the dual exploration process $({\hat{\rho}}_t)_{t \geq 0}$ introduced in \eqref{eqn:dual exp process} which informally encodes the sizes of hubs on the left hand side of the tree branches. We define the shifted version ${\hat{\rho}}^{[L,T]}$ as a measure on $[0, \delta]$ as follows. Recall also that $(\zeta_s)_{s \geq 0}$ denotes the height process of $\Ta$. For $A \subset [0, \delta]$ define $A^L = \zeta_L + A = \{x \in (0, \infty): x - \zeta_L \in A\}$. Then set
\[
{\hat{\rho}}^{[L,T]} (A) = {\hat{\rho}}_T (A^L).
\]
It follows by construction that in fact ${\hat{\rho}}^{[L,T]}$ is the projection of ${\hat{\rho}}_T$ to the part of the branch $P_T$ from the root of the tree to $p(T)$ coded by the time interval $[L,T]$, indexed by the length of this branch (which is $\delta$).

\begin{figure}[h]
\includegraphics[width=7cm]{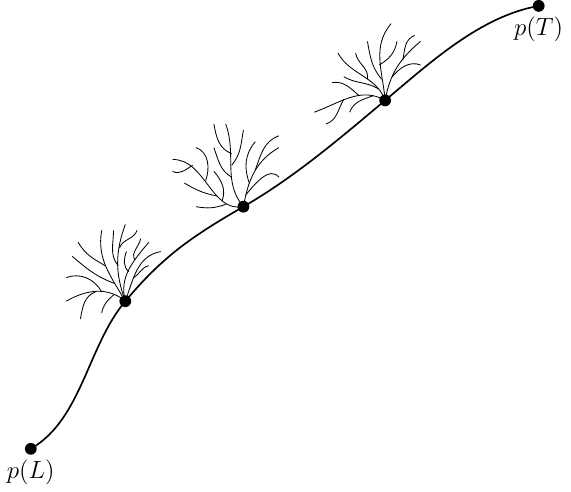}
\centering
\caption{The sizes of the hubs on the left of this branch are encoded by ${\hat{\rho}}^{[L,T]}$, and the subtrees emanating from them are coded by the \Ito measure, with an extra restriction on their height.}\label{fig:subtrees penalised}
\end{figure}

We let $P_{[L,T]}$ denote the branch between $p(L)$ and $p(T)$ and we continue to index it by the distance from $p(L)$ (recall that its total length is $\delta$). The interval $[L,T]$ codes precisely this branch and the subtrees grafted to the left of it. In fact this collection of subtrees form a Poisson point process, the law of which can be deduced from \cite[Theorem 4.6.2]{LeGDuqMono} and \cite[Theorem 3.1]{abraham2009williams}; we explain this in the following few paragraphs. To this end, we define the measure
\begin{equation}\label{eqn:PP for subtrees}
\Gamma = \sum_{i \in \I} \delta_{(x_i, \T_i)},
\end{equation}
where $\T_i$ denotes the $i^{th}$ tree in the collection and $x_i$ denotes the location of its root on the branch $P_{[L,T]}$ (indexed by distance to $p(L)$).

Consider the shifted time-reversed height process $\overleftarrow{H}_t = \zeta_{T-t}-\zeta_L$ and associated \Levy process $\overleftarrow{X}_t = X_{(T-t)^-}$ for $t \in [0, T-L]$. Note that, by construction, $\HH_0=\delta$ and that $\HH_t < \delta$ for all $t>0$. Moreover, if $(\rho_t)_{t \in [0,T-L]}$ denotes the exploration process associated with $\overleftarrow{X}$, then necessarily $\rho_0 = {\hat{\rho}}^{[L,T]}$.

It therefore follows from \cite[Theorem 3.1]{abraham2009williams} that conditionally on ${\hat{\rho}}^{[L,T]}$, the point measure in \eqref{eqn:PP for subtrees} is Poisson with intensity 
\begin{equation}\label{eqn:PP Williams}
{\hat{\rho}}^{[L,T]}(dr) N (\cdot, H(\T) \leq \delta - r). 
\end{equation}
In addition, it follows from \cite[Theorem 4.6.2]{LeGDuqMono} that ${\hat{\rho}}^{[L,T]}$ has law corresponding to that of
\begin{equation}\label{eqn:eta PP}
\sum_{j \in \mathcal{J}} (v_j - \ell_j) \delta_{r_j},
\end{equation}
where $\sum_{j \in \mathcal{J}}\delta (r_j, v_j, \ell_j)$ is a Poisson point process with intensity
\begin{equation}\label{eqn:PP intensity with height condition}
C_{\alpha}\indic{[0, \delta]}(r) \indic{[0,v]}(\ell) e^{-c_{\alpha}\ell (\delta - r)^{-\frac{1}{\alpha - 1}}} dr v^{-\alpha} d\ell dv
\end{equation}
and $C_{\alpha}$ is given by \eqref{eqn:psi Levy Khinchin} and $c_{\alpha} = (\alpha-1)^{-\frac{1}{\alpha-1}}$ by \cite[Proposition 5.6]{goldschmidt2010behavior}. (To see this, take $W$ in \cite[Theorem 4.6.2]{LeGDuqMono} to be the identity map $\Wt_t = \zeta_t$, rather than the standard Brownian snake - note that the theorem only requires that the snake process is continuous.)

This means that we can analyse the Poisson process $\Gamma$ appearing in \eqref{eqn:PP for subtrees} by applying \eqref{eqn:PP Williams} and then averaging over ${\hat{\rho}}^{[L,T]}$. We then have to add our Brownian snake process back into the mix. Since the times $L$ and $T$ depend only on the tree structure and since conditionally on the tree, the snake is just a Gaussian process with independent increments, we can sample $W$ conditionally on $\Gamma$. This is equivalent to the following: first sample ${\hat{\rho}}^{[L,T]}$ according to \eqref{eqn:eta PP}, then sample $W$ restricted to $P_{[L,T]}$ as a linear Brownian motion on $(\xi_t)_{t \in [0, \delta]}$, started at $\Wt_L$. Then, conditionally on these, sample a Poisson process $\sum_{i \in \I} \delta(x_i, (\rho_i, W_i))$ with intensity measure
\[
{\hat{\rho}}^{[L,T]}(dr) \N_{\xi_t} (\cdot, H(\T) \leq \delta - r). 
\]
Note the analogy with the calculation appearing in the proof of Proposition~\ref{prop:momentI}: the only difference with the calculation there is the extra exponential factor in \eqref{eqn:PP intensity with height condition} and height restriction appearing in \eqref{eqn:PP Williams}, which arise as a result of our conditioning.

To synchronise our notation with that of previous calculations, we let $E^0$ denote the law of $\hat{\rho}^{[L,T]}$, $\Pi$ denote the law of the linear Brownian motion $\xi$ on the branch from $p(L)$ to $p(T)$, and $\mathbb{E}_{\Pcal}$ the law of the Poisson process of \eqref{eqn:eta PP}. We use the above construction to calculate as follows, applying \cref{prop:Poisson master formula} in the final line: 
\begin{align}\label{eqn:increase points PP exp}
\begin{split}
&\mathbb{P}^* \left(\Wt_s > \Wt_L - \eta \text{ for all } s \in [L,T] \middle| T<T' \right) \\
 &\qquad = E^0 \otimes \Pi_{\widehat{W}_L} \left[  \indic{\xi ([0, \delta]) \subset (\widehat{W}_L -\eta, \infty)}  \estart{\prod_{i \in \I} \indic{\underline{W} \geq \widehat{W}_L -\eta} \middle| (\hat{\rho}^{[L,T]},\xi)}{\mathcal P} \right] \\
  &\qquad = E^0 \otimes \Pi_{\eta} \left[  \indic{\xi ([0, \delta]) \subset ( 0, \infty)} \estart{\prod_{i \in \I} \indic{\underline{W} \geq 0 }\middle| (\hat{\rho}^{[L,T]},\xi)}{\mathcal P}\right] \\
&\qquad =  E^0 \otimes \Pi_{\eta} \left[ \indic{\xi ([0, \delta]) \subset (0, \infty) }\exp \left\{ -\int_0^{\delta} \N_{\xi_r}(\underline{W} < 0, H<\delta - r) {\hat{\rho}}^{[L,T]} (dr) \right\}  \right].
\end{split}
\end{align}

By the L\'evy-Khinchin formula (and taking appropriate limits to deal with the inhomogeneities) and setting $f(r,\delta) =  \N_{\xi_r} \left( H < \delta - r, \underline{W} < 0 \right)$ (this is positive when $\xi_r>0$), we have that 
\begin{align}\label{eqn:Laplace transform inhomogen}
E^0 \left[ \exp \left\{ -\int_0^{\delta} f(r,\delta) {\hat{\rho}}^{[L,T]} (dr) \right\}\right] = \exp \left\{ -\int_0^{\delta} \psi^{\delta-r} \left(f(r,\delta)\right)  dr \right\},
\end{align}
where, for $x,y >0$,
\begin{align*}
\psi^{y}(x) &= C_{\alpha}\int_0^{\infty} \left( \int_0^u u^{-1} (1-e^{-\ell x}) e^{-c_{\alpha} \ell y^{-\frac{1}{\alpha-1}}} d\ell \right) u^{-\alpha} du.
\end{align*}
For comparison, if we did not have the exponential term in \eqref{eqn:PP intensity with height condition} the relevant function would be given by the same integral but without the second exponential term, which is just the function $\tilde{\psi}$ (check \eqref{eqn:subordinator uniform split}). A simple calculation gives a nice lower bound for $\psi^y$:
\begin{align*}
\psi^y(x) &= C_{\alpha}\int_0^{\infty} \left( \int_0^u u^{-1} \left( e^{-c_{\alpha} \ell y^{-\frac{1}{\alpha-1}}} -e^{-\ell (x + c_{\alpha} y^{-\frac{1}{\alpha-1}})} \right) d\ell \right) u^{-(\alpha-1)} du\\
& = \widetilde{\psi} \left( x + c_{\alpha} y^{-\frac{1}{\alpha-1}} \right) - \widetilde{\psi} \left( c_{\alpha} y^{-\frac{1}{\alpha-1}} \right)\\
& \geq \widetilde{\psi} \left( x \right) -  c_{\alpha}^{\alpha -1} y^{-1}.
\end{align*}
In turn, the right hand side of \eqref{eqn:Laplace transform inhomogen} is upper bounded by
\begin{align*}
\exp \left\{ -\int_0^{\delta} \psi^{\delta} \left(f(r,\delta)\right)  dr \right\} &\leq \exp \left\{ -\int_0^{\delta/2} (\tilde{\psi} \left(f(r,\delta)\right) - c_{\alpha}^{\alpha -1} (\delta -r)^{-1}) dr \right\} \\
&= \exp \left\{ -\int_0^{\delta/2} \tilde{\psi} \left(f(r,\delta)\right) dr + c_{\alpha, \delta}'\right\},
\end{align*}
where the constant $c'_{\alpha, \delta}$ is finite and depends only on $\alpha$ and $\delta$. We can then lower bound the integral appearing in the last display by 
\begin{align}\label{eqn:psi eta analysis increase points}
\begin{split}
\int_0^{\delta/2} \tilde{\psi}\left(\N_{\xi_{r}} \left( H < \delta/2, \underline{W} < 0\right)\right) dr &\geq \int_0^{\delta/2} \left(\N_{\xi_{r}} \left( \underline{W} < 0 \right) - \N_{\xi_{r}} \left( H \geq \delta/2\right) \right)^{\alpha-1} dr \\
&\geq \int_0^{\delta/2} \N_{\xi_{r}} \left( \underline{W} < 0 \right)^{\alpha-1} - \N_{\xi_{r}} \left( H \geq \delta/2\right)^{\alpha-1} dr \\
&= \int_0^{\delta/2} \N_{\xi_{r}} \left( \underline{W} < 0\right)^{\alpha-1} dr - C_{\alpha,\delta}'\\
& = \frac{\alpha + 1}{(\alpha -1)^2} \int_0^{\delta/2} {\xi_{r}}^{-2} dr - C_{\alpha,\delta}',
\end{split}
\end{align}
where we used Equation \eqref{eq:rangesnake} for the last line, and where $C_{\alpha,\delta}'$ is another finite constant depending on $\alpha$ and $\delta$ only.

Finally we deduce using \cref{prop:Bessel} that there are constants $c_{\alpha,\delta}, C_{\alpha,\delta}<\infty$ such that 
\begin{align}\label{eqn:Bessel abs cont increase points}
\begin{split}
&\Pi_{\eta}\left[\indic{\xi ([0, \delta]) \subset (0, \infty)} \exp \left\{ - \int_0^{\delta} \tilde{\psi}\left(\N_{\xi_t} \left( H < \delta - t, \underline{W} < 0\right)\right) dt \right\}\right] \\
&\leq c_{\alpha,\delta} \Pi_{\eta}\left[\indic{\xi ([0, \delta]) \subset (0, \infty)} \exp \left\{ - \frac{\alpha + 1}{(\alpha -1)^2} \int_0^{\delta/2} \xi_t^{-2}\right\}\right] \\
&= C_{\alpha,\delta} \eta^{\tilde{\nu} + 1/2}E_{\eta}\left[{(R_{\delta/2})^{-\tilde{\nu}-\frac{1}{2}}}\right] \leq C_{\delta} \eta^{\tilde{\nu} + 1/2}E_0\left[(R_{\delta/2})^{-\tilde{\nu}-\frac{1}{2}}\right],
\end{split}
\end{align}
where $(R_t)_{t \geq 0}$ is a Bessel process of dimension ${2+2\tilde{\nu}}$ and
\[
\tilde{\nu} = \sqrt{2\frac{\alpha + 1}{(\alpha -1)^2} + \frac{1}{4}} = \frac{(\alpha + 3)}{2(\alpha -1)}.
\]
The final inequality in \eqref{eqn:Bessel abs cont increase points} follows from the stochastic monotonicity between squared Bessel processes started from $\eta$ and $0$. Since $P({R_{\delta/2} < x}) = O (x^{2 + 2\tilde{\nu}})$ for all sufficiently small $x$ (e.g. see \cite[Equation (15)]{lawler2018notes}) the final expectation in \eqref{eqn:Bessel abs cont increase points} is finite and in particular does not depend on $\eta$.
\end{proof}

Using scaling properties of the \Ito measure, we deduce the following.

\begin{corollary}\label{cor:inc points lem normalised}
Under $\Noo$ choose $U$ uniformly on $[0, 1]$, take $\delta \in (0,1)$ and set $T=T_U, T'=T_U'$ and $L=L_U$. Then there exists $C_{\delta} < \infty$ such that for all $\eta \in (0,1)$ and all $x \in \R$:
\[
\Noo_x\left(T<T' \text{ and } \Wt_s > \Wt_L - \eta \text{ for all } s \in [L,T]\right) \leq C_{\delta} \epsilon \eta^{\frac{\alpha+1}{\alpha-1}}.
\]
\end{corollary}

We are now ready to prove~\cref{thm:no common increase points intro}.

\begin{proof}[Proof of \cref{thm:no common increase points intro}]
Select $U \sim \textsf{Uniform} ([0,1])$ and take some small $\kappa \in (0, \frac{\alpha + 1}{2}-1)$. For each $1 \leq i \leq \lceil \epsilon^{-\frac{\alpha + \kappa}{\alpha-1}} \rceil$ set 
\[
U_i = (U+i\epsilon^{\frac{\alpha + \kappa}{\alpha-1}}) \pmod 1.
\]
For each $i$ we also define
\begin{align*}
T_i = \inf \{s \geq U_i : \zeta_s= \zeta_{U_i} + \delta \}, \ \ T_i' = \inf \{s \geq U_i : \zeta_s= (\zeta_{U_i} - \epsilon ) \vee 0 \}
\end{align*}
(for now we suppress the dependence on $\epsilon$ in the notation). On the event $\{T_i< T_i'\}$ set
\[
L_i = \sup \{ s < T_i: \zeta_s = \zeta_{U_i}\}.
\]
Let $A_{\epsilon, \eta}$ be the event that there exists $i$ such that $T_i < T_i'\wedge \infty$ and 
\[
\Wt_s > \Wt_{L_i} -\eta
\]
for all $s \in [L_i, T_i]$. By \cref{cor:inc points lem normalised} and a union bound (note that the marginal of each $U_i$ is uniform on $[0,1]$), we have that
\begin{align*}
\Noo (A_{\epsilon, \eta}) &\leq \sum_{i=1}^{\lceil \epsilon^{-\frac{\alpha + \kappa}{\alpha-1}} \rceil} \Noo \left(T_i < T_i', \, \Wt_s > \Wt_{L_i} -\eta \quad \forall s \in [L_i, T_i]\right) \\
& \leq 2\epsilon^{-\frac{\alpha + \kappa}{\alpha-1}} C_{\delta} \epsilon \eta^{\frac{\alpha + 1}{\alpha -1}}\\
&  = 2\epsilon^{-\frac{(1+\kappa)}{\alpha-1}} C_{\delta} \eta^{\frac{\alpha + 1}{\alpha -1}}.
\end{align*}
Now choose $\frac{1+\kappa}{\alpha + 1} < b < \frac{1}{2}$ (this is possible by our choice of $\kappa$) and define a sequence $(\epsilon_p, \eta_p)_{p \geq 1}$ by setting $\epsilon_p = 2^{-p}$ and $\eta_p = (4\epsilon_p)^b$. It then follows from the first Borel-Cantelli lemma that $\Noo(A_{\epsilon_p, \eta_p} \text{ i.o.})=0$.

As in \cite[proof of Lemma 2.2]{LeGall2007BrownianMapTopological}, we now use the following observation to complete the argument. To prove the lemma, it is enough to show that there cannot exist $\delta>0$ and $r \geq 0$ such that $\inf \{s \geq r : \zeta_s= \zeta_r + 2\delta \}< \infty$ and 
\begin{equation}\label{eqn:event for increase points}
\zeta_s \geq \zeta_r \text{ and } \Wt_s \geq \Wt_r \text{ for every } s \in [r, \inf \{ u \geq r : \zeta_u = \zeta_r + 2 \delta\}].
\end{equation}
Note also that the above event is monotone in $\delta$ so it is sufficient to prove that for a fixed (small) $\delta$ there does not exist such an $r \in [0,1]$. We also use \cref{lem:height holder} which says that the height function $\zeta$ is $\Noo$ almost surely H\"older continuous with exponent $\gamma$ for any $\gamma < 1-\frac{1}{\alpha}$. In particular, since $|U_i - U_{i+1}| \pmod 1 = \epsilon^{\frac{\alpha + \kappa}{\alpha -1}}$, it follows that for all sufficiently small $\epsilon$,
\begin{equation}\label{eqn:Holder consequence}
\sup_{s,t \in [U_i, U_{i+1}]}|\zeta_{s} - \zeta_{t}| \leq \epsilon
\end{equation}
for all $i$, $\Noo$-almost surely (here for the index $i$ where $U_{i+1}< U_i$, the interval in the expression should be interpreted as $[U_i , 1] \cup [0, U_{i+1}]$).

\begin{figure}[h]
\includegraphics[height=5cm]{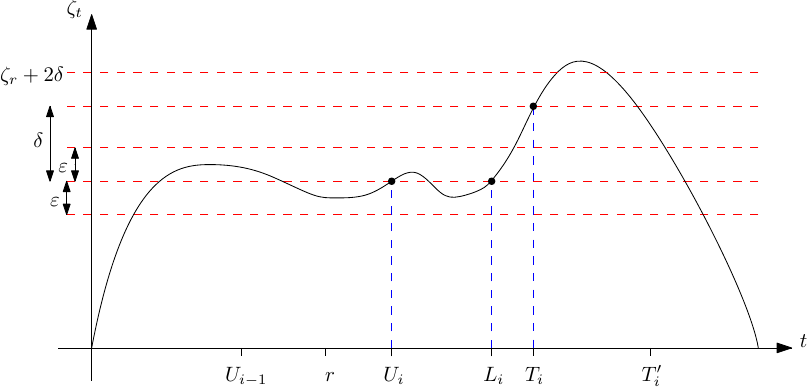}
\centering
\caption{Relation between times appearing in the event \eqref{eqn:event for increase points}. Not to scale.}\label{fig:noincrease} 
\end{figure}

Now fix $\delta>0$ and suppose that there exists $r >0$ realising the event in \eqref{eqn:event for increase points}. 
 In what follows we reintroduce the dependence on $\epsilon$ into the notation for $U_i, T_i, T_i'$ through superscripts. For each $p \geq 1$ take $i$ such that $r \in (U^{\epsilon_p}_{i-1}, U^{\epsilon_p}_{i}]$ (note that since we assumed that $r \neq 0$ and it is not possible that $r=1$, we will have $0 < \epsilon^{\frac{\alpha + \nu}{\alpha-1}} < r < 1- \epsilon^{\frac{\alpha + \nu}{\alpha-1}} < 1$ for all sufficiently large $p$, so this interval is well-defined). Note that on the event $T_i < T_i'$ it follows that $ \dt(p(U^{\epsilon_p}_i), p(L^{\epsilon_p}_i)) \leq 2\epsilon$ by \eqref{eqn:distance from height}. Hence combining with \eqref{eqn:Holder consequence} gives that, if $p$ is large enough (i.e. $\epsilon_p$ is small enough) and the event $T_i^{\epsilon_p} < T^{\epsilon_p'}_i$ occurs,
\[
\dt(p(r), p(L^{\epsilon_p}_i)) \leq \dt(p(r), p(U^{\epsilon_p}_i)) + \dt(p(U^{\epsilon_p}_i), p(L^{\epsilon_p}_i)) \leq 4\epsilon_p.
\]
Moreover \eqref{eqn:Holder consequence} also implies (again if $p$ is large enough) that $T^{\epsilon_p}_{i} \leq \inf \{ u \geq r : \zeta_u = \zeta_r + 2 \delta\}$, and \eqref{eqn:event for increase points} and \eqref{eqn:Holder consequence} together imply that $T^{\epsilon_p}_i < T^{\epsilon_p'}_i$.

Hence on the event in \eqref{eqn:event for increase points} and provided that $p$ is sufficiently large, it holds for all $s \in [U^{\epsilon_p}_i, T^{\epsilon_p}_i]$ that
\[
\Wt_s \geq \Wt_r \geq \Wt_{L^{\epsilon_p}_i} - (4\epsilon_p)^b = \Wt_{L^{\epsilon_p}_i} - \eta_p.
\]
Here the last line follows by \cref{lem:height holder}(b) since $\dt(p(r), p(L^{\epsilon_p}_i))<4\epsilon_p$, and therefore contradicts the fact that $A_{\epsilon_p, \eta_p}$ cannot happen infinitely often, $\Noo$-almost surely.

To rule out the case $r=0$, we can re-root the tree at at a uniform point. By \cref{prop:uniform rerooting} this then reduces to the case above.
\end{proof}

\newpage

\addcontentsline{toc}{section}{Table of notation}
\section*{Table of notation}\label{sec:notations}

\subsection*{Stable Trees}
\begin{longtable}{p{3.5cm}p{12cm}}
$(\Ta, \dt)$ & stable tree with distance $\dt$  \\
$X$ & spectrally positive stable Lévy process with index $\alpha$\\
$\Delta_s$ & Equal to $\X_s - \X_{s^-}$, the size of the jump at time $s$ \\
$\psi$ & branching mechanism \\
$C_{\alpha}, \pi$ & opposite of drift coefficient and jump measure of $X$\\
$I_t, I_{s,t}$ & running infimum of $X$, infimum of $X$ on $[s,t]$ \\
$\X$ & normalised excursion of $X$ above its infimum at time 1\\
$N$ & \Ito excursion measure for excursions (and trees) \\
$H$ & height process\\
$(\rho_t, {\hat{\rho}}_t)_{t\geq 0}$ & exploration process and dual exploration process \\
$\bPb_{\mu}(d \rho)$ & law of exploration process started from distribution $\mu$ \\
$p$ & canonical projection $[0,1] \to \Ta$\\
\end{longtable}

\subsection*{Stable Snakes}
\begin{longtable}{p{3.5cm}p{12cm}}
$(Z_v)_{v \in \Ta}$ & stable snake, indexed by $\Ta$ \\
$\W$ & set of continuous real-valued functions on interval of the form $[0,\zeta]$\\
$\zeta$ & lifetime of the snake\\
$w$, $\widehat{w}=w(\zeta_w)$ &  element of $\W$, and its tip\\
$(\xi_t)_t$ & standard linear Brownian motion\\
$\Pi_x$ & law of standard linear Brownian motion started at $x$\\
$Q_{w_0}^h$ & snake on forest with height function $h$, with initial condition $w_0$\\
$\Pb_{\mu,w}(d \rho, dW)$ & law of stable snake with initial condition $(\mu,w)$\\
$\big((W_s(t))_{t \leq \zeta_s}\big)_{s \in [0,1]}$ & snake indexed by time \\
$(\widehat{W}_t)_{t \in [0,1]}$ & $(W_t(\zeta_t))_{t \in [0,1]}$, the tip of the snake indexed by time \\
$\Pb_{x}$ & law of stable snake with initial condition $(0,x)$ \\
$\Pb^*_{\mu,w}(d \rho, dW)$ & law of $(\rho,W)$ under $\Pb_{\mu,w}$, killed when $\rho$ first hits 0\\
$\N_x$ and $\Noo_x$ & excursion measures for stable snakes\\
$\zetas, \Ws, \Wts$ & lifetime, path and tip of the snake $W$ at $s$\\
$(U^{(1)}, U^{(2)})$ & two-dimensional subordinator giving the law of the exploration and dual exploration processes along a branch to a uniform time point\\
 $(\Omega_0,\mathcal F_0, P^0)$ & probability space on which $(U^{(1)}, U^{(2)})$ is defined\\
$E^0$ & expectation associated to $(\Omega_0,\mathcal F_0, P^0)$ \\
$\tilde{\psi}$ & Laplace exponent of the marginals $U^{(1)}$ and $U^{(2)}$\\
$\psi'$ & Laplace exponent of the sum $U^{(1)}+U^{(2)}$\\
$\tilde{\pi}, \pi'$ & jump measures respectively associated to $\tilde{\psi}$ and $\psi'$\\
$(J_a, \hat{J}_a)$ & pair of random measures:  \\
& $(J_a, \hat{J}_a)(dt)=(\indic{[0,a]}(t) dU_t^{(1)}, \indic{[0,a]}(t) dU_t^{(2)})$\\
$\mathcal{R}$ & range of the stable snake \\
$\Pcal$ & Poisson point process with intensity $(J_h + \hat J_h)(da) \otimes \mathbb N_{\xi_a} (dw)$, describing the exploration and dual exploration of the tree and snake along the branch\\
$\mathbb{E}_{\Pcal}$ & expectation w.r.t Poisson point process $\Pcal$\\
$\tau^D(w)$ & first exit time of domain $D$ for $w$ \\
$\theta^D_s$ & time change on snake trajectories before their exit time from $D$ \\
$ \widetilde{W}^D_s = W_{\theta^D_s}$ & time-changed snake\\
$L^D_{s}$ & exit local time from $D$\\
$\Z^D$ & exit measure from $D$\\
$s^*$ & unique time where $W$ attains its minimum \\
$L_s$ & exit local time from $(a,b)$ until time $s$\\
$\tilde{L}_s$ & additive functional of the snake: \\
&$d \widetilde L_s = \lambda \indic{\{\tau^{(a,b)}(W_s) = +\infty\} } ds  + \left( \lambda^{1/\alpha} + \mu \right) dL_s$\\
$v_{\lambda,\mu,a,b} (x)$ & $ \mathbb N_x \left[ 1 - \exp(-\lambda \sigma - \mu L_\sigma)\right]$ \\
$v_{\lambda,\mu} (x)$ &  $v_{\lambda,\mu,0,+\infty}(x)$\\
$v_{\lambda} (x)$ &$\lim_{\mu\to+\infty} v_{\lambda,\mu,0,+\infty}(x)=\mathbb N_x \left[ 1 - \indic{0 \notin \mathcal{R}} \exp(-\lambda \sigma)\right]$ \\
$\I, \overline{\I}$ & occupation measure of the snake $Z$ and the shifted snake $\ZZZ$ \\
\end{longtable}

\subsection*{Misc}
\begin{longtable}{p{3.5cm}p{12cm}}
$\R_+$      & $[0,+\infty)$\\
$(R_t)_{t \geq 0}$, $P_x^{(d)}$, $\mathbb{E}_x^{(d)}$ &  Bessel process of dimension ${2+2\tilde{\nu}}$, its law and expectation \\
$\tilde{\lambda}, \tilde{\nu}$ & $\tilde{\lambda}^2 = \frac{2 \alpha (\alpha +1)}{(\alpha -1)^2}$, $\tilde{\nu} = \frac{3 \alpha +1}{2(\alpha -1)}$ \\
$M_f(\R_+)$ & set of finite measures on $\R_+$\\

\end{longtable}


\addcontentsline{toc}{section}{References}
\printbibliography

\bigskip
\bigskip

{Eleanor Archer and Laurent Ménard:}

{Modal'X, UMR CNRS 9023, UPL, Univ. Paris-Nanterre, F92000 Nanterre, France.}  

\medskip

{Ariane Carrance:}

{CMAP, UMR CNRS 7641, \'Ecole polytechnique, F9112 Palaiseau, France.}

\end{document}